\documentclass[10pt,a4paper]{article}

\pdfoutput=1

\usepackage{amssymb}
\usepackage{amsmath}
\usepackage{amsthm}
\usepackage{xcolor}
\usepackage{graphicx}
\usepackage{epstopdf}
\usepackage{enumitem}
\usepackage{tabularx}
\usepackage[ruled,algosection]{algorithm2e}
\SetKwRepeat{Repeat}{Do}{While}
\usepackage{comment}
\usepackage{psfrag}
\usepackage{hyperref}
\hypersetup{
  breaklinks,bookmarks=false,colorlinks = true,
             linkcolor = blue,
             urlcolor  = blue,
             citecolor = blue,
             anchorcolor = blue
}
\usepackage[textsize=tiny,color=yellow!60!green]{todonotes}
\usepackage{mathbbol}
\usepackage{amssymb}             
\usepackage{enumitem}
\usepackage{subcaption}
\DeclareSymbolFontAlphabet{\amsmathbb}{AMSb}%
\newcommand{\textn}[1]{\textnormal{#1}}
\newcommand{\bP}{\amsmathbb{P}}

\DeclareMathAlphabet{\mathbfit}{OML}{cmm}{b}{it}

\newcommand{\RR}{{\amsmathbb R}}

\newcommand{\vecv}{{\mathbf{v}}}
\newcommand{\vecu}{{\mathbf{u}}}
\newcommand{\veca}{{\mathbf{a}}}
\newcommand{\vecw}{{\mathbf{w}}}
\newcommand{\vecz}{{\mathbf{z}}}
\newcommand{\vK}{\mathbf K}
\newcommand{\II}{{\mathbf I}}
\newcommand{\dt}{\,\textn{d}t}
\newcommand{\dx}{\,\textn{d}\vecx}

\newcommand{\vectau}{{\boldsymbol {\tau}}}
\newcommand{\vecsigma}{{\boldsymbol {\sigma}}}
\newcommand{\vecgamma}{{\boldsymbol {\gamma}}}
\newcommand{\veczeta}{{\boldsymbol {\zeta}}}
\newcommand{\vectheta}{{\boldsymbol {\theta}}}

\newcommand{\vf}{\mathbf f}

\newcommand{\epsb}{\boldsymbol{\epsilon}}
\newcommand{\vecn}{{\mathbf{n}}}
\newcommand{\vecx}{{\mathbf{x}}}

\newcommand{\divecv}{\nabla{\cdot}}

\newcommand{\divecvv}{{\textn{div}}}

\newcommand{\omegatimeszeroT}{\Omega\times(0,T)}

\newcommand{\Hdiv}{\mathbf{H}(\divecvv,\Omega)}

\newcommand{\Htensordiv}{\amsmathbb{H}(\divecvv,\Omega)}

\newcommand{\htau}{h\tau}

\newcommand{\Iavo}{\mathcal{I}_{\textn{av}}}

\newcommand{\eq}{:=}

 \newcommand{\sigmabo}{\boldsymbol{\sigma}}
 
  \newcommand{\betabo}{\boldsymbol{\beta}}

\newcommand{\Tau}{\mathcal{T}}
\newcommand{\Vauh}{\mathcal{V}_{h}}

\newcommand{\RTN}{{\mathbf{RTN}}}

\newtheorem{thm}{Theorem}
\numberwithin{thm}{section}

\newtheorem{lem}[thm]{Lemma}

\newtheorem{df}[thm]{Definition}

\newtheorem{rem}[thm]{Remark}

\newtheorem{algo}[thm]{Algorithm}
\newtheorem{deff}[thm]{Definition}

\newcommand{\bse}{\begin{subequations}}
\newcommand{\ese}{\end{subequations}}

\makeatletter
\newcommand{\eqnum}{\refstepcounter{equation}\textup{\tagform@{\theequation}}}
\makeatother
\graphicspath{{Figures/}}

\title{Adaptive asynchronous time-stepping,  stopping criteria, and a posteriori error estimates for fixed-stress iterative  schemes   for  coupled poromechanics problems\thanks{This work forms part of Norwegian Research Council project 250223}}

\author{Elyes Ahmed\footnotemark[2]
\and Jan Martin Nordbotten\footnotemark[2]\ \footnotemark[3]
\and Florin Adrian Radu\footnotemark[2]
}
\date{\today}

\begin{document}

\maketitle

\renewcommand{\thefootnote}{\fnsymbol{footnote}}

\footnotetext[2]{Department of Mathematics, University of Bergen, P. O. Box 7800, N-5020 Bergen, Norway.
\href{mailto:elyes.ahmed@uib.no}{elyes.ahmed@uib.no},
\href{mailto:jan.nordbotten@uib.no}{jan.nordbotten@uib.no},
\href{mailto:florin.florin.radu@uib.no}{florin.radu@uib.no},
}
\footnotetext[3]{Department of Civil and Environmental Engineering, Princeton University, Princeton, N. J., USA.}
\renewcommand{\thefootnote}{\arabic{footnote}}

\numberwithin{equation}{section}

\begin{abstract}
In this paper we develop   adaptive iterative coupling schemes for the Biot system modeling coupled poromechanics problems. We  particularly consider the   space-time formulation of the fixed-stress iterative scheme,
in which  we  first solve  the problem of flow   over the whole space-time interval,   then exploiting the space-time information for  solving the   mechanics. Two common discretizations of this algorithm are  then introduced  based on two coupled mixed finite element methods in-space and  the    backward Euler scheme in-time. Therefrom, adaptive fixed-stress algorithms  are  build on conforming reconstructions of the pressure and displacement together with equilibrated flux and stresses reconstructions. These ingredients are used to  derive a posteriori error  estimates for the fixed-stress algorithms, distinguishing the different error components, namely the spatial discretization, the temporal discretization, and the fixed-stress iteration components. Precisely, at the  iteration $k\geq 1$ of the adaptive  algorithm, we prove that our estimate gives a guaranteed  and fully 
computable upper bound    on the energy-type error measuring the difference between the exact and  approximate   pressure and displacement. These error components are    efficiently used to design adaptive  asynchronous time-stepping and adaptive  stopping criteria for the  fixed-stress algorithms. 
Numerical experiments illustrate  the efficiency of our estimates and the performance of the adaptive  iterative coupling 
algorithms. 
\end{abstract}

\vspace{3mm}



\pagestyle{myheadings} \thispagestyle{plain} \markboth{E. Ahmed and F. A. Radu}{}

%
\section{Introduction} \label{sec:Intro}
 Let $\Omega$ be  an open, bounded   and connected  domain in $\RR^{d}$, $d=2,3$,
 which is assumed to be polygonal with Lipschitz-continuous boundary $\partial\Omega$,  and let 
 $T$ be the final simulation time. We consider in this paper 
 the problem of  flow in deformable porous media modelled by  the \textit{quasi-static Biot system} \cite{SHOWALTER2000310}:
find a displacement $\vecu$ and 
a pressure $p$ satisfying:
\bse\label{Main_problem_model}\begin{alignat}{3}
  -\divecv(\vectheta(\vecu)-\alpha p\II) &=\vf,\quad&&\mbox{in }\omegatimeszeroT,
    \label{Main_problem_model_mech}\\
  \partial_{t}\varphi(p,\vecu)-\divecv(\vK\nabla p)&=g,\quad&&\mbox{in }\omegatimeszeroT, \label{Main_problem_model_cons}\\   
\vecu(\cdot,0)=\vecu_{0},\quad p(\cdot,0)&=p_{0},\quad&&\textnormal{in }\Omega, \label{Main_problem_model_IC}\\
   \vecu=0,\quad p&=0,\quad&&\mbox{on }\partial\Omega\times(0,T), \label{Main_problem_model_BC}
\end{alignat}\ese
where $\vf$ is the body force and $g$ is the volumetric source term. 
The function $\vectheta$ denotes the effective stress tensor, i.e.,  $\vectheta(\vecu)\eq2\mu\epsb(\vecu)+\lambda\textn{tr}(\epsb(\vecu))\II $, 
with  $\epsb(\vecu)$ is the linearized strain tensor  given by $\epsb(\vecu)\eq(\nabla \vecu+\nabla^{\textn{T}}\vecu)/2$ and
the operator \textquotedblleft $\textn{tr}$\textquotedblright\,  denotes  the trace of matrices.  The coefficients  $\mu$ and $\lambda$ are
the Lam\'{e} parameters,   supposed strictly positive constants. The function $\varphi$ denotes 
the fluid content, i.e.,  $\varphi(p,\vecu):=c_{0}p+\alpha\divecv\vecu$,  where 
$c_{0}>0$   is the constrained-specific storage coefficient, and $\alpha>0$  is 
the Biot--Willis constant. The parameter $\vK$ is the permeability tensor divided by fluid viscosity; it is  
 a symmetric, bounded, and uniformly positive definite tensor whose
terms are for simplicity supposed piecewise constant on the mesh $\Tau_{h}$ 
of $\Omega$ defined below and constant in time. Finally,     $p_{0}$ is the initial pressure and 
$\vecu_{0}$ is the initial displacement. 

Throughout, we will use the convention that if
$V$ is a space of functions, then we designate by  $\mathbf{V}$  a space of vector functions having each component
in $V$, and we designate by $\amsmathbb{V}$  the  space of tensor functions having each component
in $V$. Let $D\subset \RR^{d}$; the space $L^{2}(D)$ is endowed with its natural inner product written $\left(\cdot,\cdot\right)_{D}$ 
with associated norm denoted by $||\cdot||_{D}$.  When the domain $D$
coincides with $\Omega$, the subscript $\Omega$ is dropped. Let $|D|$ be the
Lebesgue measure of $D$. We designate by $H^{1}(\Omega)$ the usual Sobolev space and by $H^{1}_{0}(\Omega)$ for its zero-trace subspace. Its norm and semi-norm are 
written $||\cdot||_{H_{0}^{1}(\Omega)}$ and $|\cdot|_{H_{0}^{1}(\Omega)}$ respectively.   In
particular, $H^{-1}(\Omega)$ is the dual of $H^{1}_{0}(\Omega)$. Further,  let $\Hdiv$ be the space of
vector-valued functions  from $\mathbf{L}^{2}(\Omega)$ that admit a
weak divergence in $L^{2}(\Omega)$. Its natural norm is 
\[
 ||\vecv||_{\divecvv,\Omega}\eq\left(||\vecv||^{2}+||\divecv \vecv||^{2}\right)^{\frac{1}{2}}.
 \]
We also  define $\Htensordiv$ to be the space of
tensor-valued functions  from $\amsmathbb{L}^{2}(\Omega)$ that admit a
 weak divergence  (by rows) in $\mathbf{L}^{2}(\Omega)$. 
Then,  we set 
\begin{align*}
  Q& \eq  L^{2}(\Omega), \qquad   \mathbf{W}\eq \Hdiv,  \qquad \amsmathbb{W}  \eq \Htensordiv,\qquad \amsmathbb{Q}_{\textn{sk}}  \eq [L^{2}(\Omega)]_{\textn{sk}}^{d\times d}.
\end{align*}
To give the mixed formulation 
of~\eqref{Main_problem_model}, we introduce 
the total stress tensor; 
$\vecsigma(p,\vecu)\eq\vectheta(\vecu)-\alpha p \II$, 
and the Darcy velocity;  $\vecw\eq-\vK\nabla p$. Let $c_{r}\eq\dfrac{d\alpha^2}{2\mu+d\lambda}$,  
then introduce the fourth-order compliance tensor $ \mathcal{A}$ given by
\begin{alignat}{2}\label{OperatorA_defini}
 \mathcal{A}\vectau\eq\dfrac{1}{2\mu}\left(\vectau-\dfrac{\lambda}{2\mu+d\lambda}\textn{tr}(\vectau)\II\right),
\end{alignat}
which is known to be bounded and symmetric definite uniformly with respect to $\vecx\in\Omega$,
we can rewrite
equations~\eqref{Main_problem_model} in mixed weak sense:
\begin{df}[The five-field formulation~\cite{ahmed:hal-01687026}]
Assume $\vf\in L^{2}(0,T;\left[L^{2}(\Omega)\right]^{d})$, $g\in L^{2}(0,T;L^{2}(\Omega))$, $p_{0}\in  H^{1}_{0}(\Omega)$  and 
$\vecu_{0}\in  \mathbf{H}^{1}_{0}(\Omega)$. The fully mixed  formulation 
of~\eqref{Main_problem_model} reads: 
find $(\vecsigma,\vecu,\vecw,p,\veczeta)\in H^{1}(0,T;\amsmathbb{W}) \times L^{2}(0,T;\mathbf{Q})\times L^{2}(0,T;\mathbf{W})\times H^{1}(0,T;Q)\times L^{2}(0,T;\amsmathbb{Q}_{\textn{sk}})$ such that
\bse\label{Weak_problem_mixed_model}\begin{alignat}{4}
\int_{0}^{T}\{(c_{0}+c_{r})(\partial_{t}p,q)
+\dfrac{c_{r}}{d\alpha}(\partial_{t}\vecsigma,q\II)+(\divecv\vecw,q)\}\dt&=\int_{0}^{T}(g,q)\dt,&&\quad\forall q\in Q,\label{Weak_problem2_model_cons}\\
\int_{0}^{T}\{(\vK^{-1}\vecw,\vecv)-(p,\divecv\vecv)\}\dt&=0,&&\quad\forall \vecv\in \mathbf{W},\label{Weak_problem2_model_darcy}\\
-\int_{0}^{T}\{(\mathcal{A}\vecsigma,\vectau)+(\vecu,\divecv\vectau)+(\veczeta,\vectau)+\dfrac{c_{r}}{d\alpha}(p\II ,\vectau)\}\dt&=0,&&\quad\forall \vectau\in \amsmathbb{W},\label{Weak_problem2_model_mech}\\
\int_{0}^{T}\{(\divecv\vecsigma,\vecz)+(\vecsigma,\vecgamma)\}\dt&=-\int_{0}^{T}(\vf,\vecz)\dt,&&\quad\forall (\vecz,\vecgamma)\in \mathbf{Q}\times \amsmathbb{Q}_{\textn{sk}},\label{Weak_problem2_model_stress}
\end{alignat}\ese
together with the initial condition~\eqref{Main_problem_model_BC}.
\end{df}
The well-posedness  and  regularity analysis of 
the Biot equations~\eqref{Main_problem_model} have been addressed in~\cite{SHOWALTER2000310}. 
That of the  existence and uniqueness of a weak solution of problem~\eqref{Weak_problem_mixed_model} 
 have been addressed in~\cite{AHMED2018} (see~\cite{ahmed:hal-01687026} for more details). Therein,  two mixed formulations are  discretized with the backward Euler scheme in-time 
 and in-space with  mixed finite elements methods, then
 a posteriori error estimates for their solutions are derived. The main  issue arising when apply \textit{MFE methods}  for this problem is that it results a very \textit{large system} to be solved at each time step (see~\cite{AHMED2018,BRUNThemo2018,yi2014convergence}). This  
issue together with the fact  that flow and mechanics effects  act at \textit{different time scales},
encourage the development of efficient techniques for  the resolution of these coupled 
systems.   \textit{Splitting-based iterative methods}~\cite{chin2002iterative,kim2009stability,mainguy2002coupling,Mikelic2014}  provide one such approach. 
They adopt the \textquotedblleft divide and conquer\textquotedblright\, strategy and  {split} the two
systems. Then, a \textit{sequential  approach} is used,  in  that either the problem of flow or the mechanics is solved 
first followed by solving the other system using the already calculated information,  
leading  to recover the original solution~\cite{KIM20112094,Mikelic2013,NAG:NAG2538}. 
{The decoupling procedure enjoys  the use of a local static condensation for the flow and mechanics. The MFE system resulting from each subsystem can be reduced to a symmetric and positive definite one;      pressure is the sole unknown for the flow problem, and  the  displacement and rotation (may also be only the displacement depending on the used quadrature rule) for the mechanics (see~\cite{ambartsumyan2018multipoint1,ambartsumyan2018multipoint2} for more details)}.  
  Particularly, in the last years, a lot of research has been done on the fixed-stress method~\cite{almani2016convergence,borregales2018parallel,DANA20181,GASPAR2017526}.
  Applied to problem~\eqref{Weak_problem_mixed_model}, it can be rewritten, as, see~\cite{BAUSE2017745}:
\begin{deff}[The space-time fixed-stress algorithm]~\label{original_fs}
{
\setlist[enumerate]{topsep=0pt,itemsep=-1ex,partopsep=1ex,parsep=1ex,leftmargin=1.5\parindent,font=\upshape}
\begin{enumerate}
\item Chose an initial approximation $\vecsigma^{0}\in H^{1}(0,T;\amsmathbb{W})$ of $\vecsigma$ 
and a tolerance $\epsilon>0$. Set $k\eq-1$.
\item \textn{\textbf{Do}}
\begin{enumerate}
 \item Increase $k\eq k+1$.
\item Compute $(\vecw^{k},p^{k})\in  
L^{2}(0,T;\mathbf{W})\times H^{1}(0,T;Q)$ such that
\bse\label{fstress_iterative_flow_optimized}\begin{alignat}{4}
\int_{0}^{T}\{(c_{0}+c_{r})(\partial_{t}p^{k},q)+(\divecv\vecw^{k},q)\}\dt&=
\int_{0}^{T}\{(g,q)-\dfrac{c_{r}}{d\alpha}(\partial_{t}\vecsigma^{k-1},q\II)\}\dt,&&\quad\forall q\in Q,\label{fstress_iterative_flow_cons_optimized}\\
\int_{0}^{T}\{(\vK^{-1}\vecw^{k},\vecv)-(p^{k},\divecv\vecv)\}\dt&=0,&&\quad\forall \vecv\in \mathbf{W}.\label{fstress_iterative_flow_darcy_optimized}
\end{alignat}\ese
\item  Compute $(\vecsigma^{k},\vecu^{k},\veczeta^{k})\in H^{1}(0,T;\amsmathbb{W})\times L^{2}(0,T;\mathbf{Q})\times L^{2}(0,T;\amsmathbb{Q}_{\textn{sk}})$ such that 
\bse~\label{fstress_iterative_mechanics_optimized}\begin{alignat}{4}
-\int_{0}^{T}\{(\mathcal{A}\vecsigma^{k},\vectau)+(\vecu^{k},\divecv\vectau)+(\veczeta^{k},\vectau)\}\dt&=\dfrac{c_{r}}{d\alpha}\int_{0}^{T}(p^{k}\II ,\vectau)\dt,&&\quad\forall \vectau\in \amsmathbb{W},\label{fstress_iterative_mech_stress_optimized}\\
\int_{0}^{T}\{(\divecv\vecsigma^{k},\vecz)+(\vecsigma^{k},\vecgamma)\}\dt&=-\int_{0}^{T}(\vf,\vecz)\dt,&&\quad\forall (\vecz,\vecgamma)\in \mathbf{Q}\times \amsmathbb{Q}_{\textn{sk}}.\label{fstress_iterative_cons_optimized}
\end{alignat}\ese
\end{enumerate}
\textn{\textbf{While}} $\dfrac{\|(\vecsigma^{k},p^{k}) - (\vecsigma^{k-1},p^{k-1})\|_{L^{2}(\Omega\times[0,T])}}{\| (\vecsigma^{k-1},p^{k-1})\|_{L^{2}(\Omega\times[0,T])}}\geq \epsilon$.
\end{enumerate}}
\end{deff}
{The above method is the space-time fixed stress 
introduced first~\cite{borregales2018parallel},  in which we  solve first  the problem of flow over the whole space-time interval,   then exchange  
 the space-time information to solve the mechanics. This method is of interest for (i) the flexibility to use  different time steps for flow and mechanics (ii)  the advantage to derive  error and a posteriori error analysis, permitting  the use of adaptive  asynchronous   time-stepping (iii) the possibility to parallelize step~2.(c) of the algorithm.}  The classical fixed-stress algorithm in-space  with four-field mixed formulation  was analyzed in~\cite{NAG:NAG2538}, where
 a priori convergence results are given. Other related works on this method can be found in~\cite{almani2016convergence,both2017robust,Castelletto2015,GASPAR2017526,Mikelic2013} and the references therein.

 In this paper, we are interested  in designing adaptive  versions of two common discretizations of the algorithm addressed in Definition~\ref{original_fs}, 
 see Algorithm~\ref{space_time_discrete},~\ref{multi-rate_discrete} (standard), and 
 Algorithm~\ref{space_time_discrete_adaptive} (adaptive) below. 
 To this aim,  two iterative solution strategies for the Biot's consolidation
problem are presented; they are based on the above  fixed stress iterative 
scheme, in which at each iteration, the space-time subsystems are solved sequentially using \textit{MFE methods} in-space and
with a \textit{backward Euler} scheme in-time (cf.~\cite{NAG:NAG2538}). We constitute their adaptive counterpart 
upon the distinction of the different error components arising in
the standard fixed stress algorithm, namely the spatial discretization, 
the temporal discretization, and the fixed stress iteration components.  To arrive to this aim, 
we take ideas from~\cite{arioli2005stopping,adaptnewernvohralik,jiranek2010posteriori,kumar2018guaranteed}, for 
general a posteriori error  techniques taking into account \textit{inexact  iterative  solvers}, but 
  most closely  form~\cite{ahmed:hal-01540956}, 
  where a domain decomposition problem is solved via space-time  iterative methods. Particularly,   
  we  will rely on~\cite[Theorem 6.2]{AHMED2018} where an energy-type-norm differences 
  between the exact and the  approximate pressure and displacement is shown to be bounded by the dual norm of the 
  residuals. {The developed adaptive fixed-stress  algorithm is applicable on  
\textit{any locally conservative discretization for the two coupled subsystems}, such as cell-centered finite volume scheme, multipoint mixed finite element, mimetic finite difference and hybrid high-order discontinuous Galerkin~\cite{MR3504993,MR3719125,MR3478962,MR3425298}. It can also be  
extended to conforming methods using  equilibrated flux and stress reconstructions~( cf. \cite{RIEDLBECK20171593}).}
  
  In contrast to what is developed 
  in~\cite{AHMED2018}, three additional features  to be treated in this work; first is that  
  the current setting targets inexact iterative coupling schemes for the Biot system and not 
  monolithic solvers; second  that the MFE methods here provides at each iteration of the 
  coupling algorithms  approximate flux and stress  not balanced with the source terms; 
  third is that the actual setting provides adaptive  asynchronous 
  time-stepping for the flow and mechanics problems. Here,   we first show that the
presented a posteriori error estimate  delivers sharp bound (as reflected by moderate effectivity indices) for the actual energy-type error, and this at each iteration of  the coupling algorithm.   We also show how  
  the overall error propagates between the flow and mechanics subproblems   during the iterative process,  and then to 
  address the question of when to stop the  iterations. 
  This question was asked in~\cite{almani2016convergence,BAUSE2017745,borregales2018parallel}, 
  where the   practitioners iterate between  the two coupled subsystems until   some fixed tolerance has been reached.
 The used  stopping criterion is in fact mostly related to the algebraic error, i.e., the closeness of 
 $(\vecsigma^{k},p^{k})$  to the convergent solution  $(\vecsigma^{\infty},p^{\infty})$ is only taken into account 
without reference to the underlying continuous Biot's  problem~\eqref{Main_problem_model}. Here, by \textit{distinguishing the space, 
  time and coupling  errors}, the  \textit{adaptive  stopping criterion} for the iterative scheme 
  that we propose instead is when  the coupling error  does not contribute 
  significantly to  the overall error. {In \textit{grosso modo}, the standard approach stops the iterations at some arbitrary tolerance, which hopefully is sufficiently accurate (but perhaps not!), while the approach based on error estimates stops the iterations at the correct time.} Adaptive stopping criteria via a posteriori error estimates in the context of 
  other model problems are treated in~\cite{ahmed:hal-01540956,hassan2017posteriori,DiPieFlaVohSol,jiranek2010posteriori}, see also the references 
  therein. 
Furthermore,  the resulting algorithms involve \textit{tuning 
parameters} that can be optimized (see~\cite{storvik2018optimization});  the results show how a posteriori error estimates  can help optimize these   parameters. To the best of our knowledge, this combination of features in the adaptive fixed-stress algorithms is unique.

The paper is organized as follows. Section~\ref{sec:notation} fix 
the notation for temporal and spatial meshes  and defines some relevant  functional spaces. In Section~\ref{sec:standard_discrete_algorithms}, we present two common discretizations of 
Algorithm~\eqref{fstress_iterative_flow_optimized}-\eqref{fstress_iterative_mechanics_optimized},   by combining in-space two mixed finite elements for the flow and mechanical problems, and  a backward Euler scheme in-time. As 
a posteriori error estimate has no meaning for piecewise constant functions,  
the MFE approximate pressure and displacement 
will  be locally postprocessed in order to obtain  improved approximations.
In Section~\ref{sec:adaptive_discrete_algorithms}, we first introduce two major 
improvements to these two standards algorithms, by designing for each one, an
adaptive stopping criterion, and a  balancing criterion  equilibrating 
the space and time error components using an 
adaptive  asynchronous time-stepping. These enhancements 
are used to design   adaptive versions of the fixed-stress schemes based on a 
posteriori error estimates. We then construct the
needed ingredients for the a posteriori error estimates:  Section~\ref{sec:reconstructions}, 
defines the   $H^{1}(\Omega)$- and $\Hdiv$ conforming reconstructions. In Section~\ref{sec:Aposteriori}, these 
 ingredients  are used
to   bound an energy-type error in the pressure and displacement  at each iteration of 
 the coupling algorithm by a guaranteed and fully computable error estimate.  This a posteriori estimate is then elaborated by   distinguishing  the fixed-stress iterative coupling error  from 
 the  space and time error components.  We also  separate 
the pressure error components from those of displacement errors. We show numerical results in Section~\ref{sec:NumericalResults}. Finally, a conclusion that highlights our developments is given in Section~\ref{sec:conclusion}.
\section{Notation}\label{sec:notation}
We introduce here the partition of $\Omega$, time discretization, notation, and function spaces; see~\cite{ahmed:hal-01687026} for a similar notation.
\subsection{Partitions of the time interval $(0,T)$}
The   space-time iterative method we use  supports asynchronous  time grids for flow and mechanics. To this aim, the
 subscripts ``$\textn{f}$'', and ``$\textn{m}$'' will be used throughout, to stand for flow and mechanics, respectively.
For integer values $N_{\textn{f}}> 0$, let $\left(\tau^{n}_{\textn{f}}\right)_{1\leq n\leq N_{\textn{f}}}$ denote a sequence of positive real 
numbers corresponding to the discrete flow time steps  such that $T=\sum_{n=1}^{N_{\textn{f}}}\tau^{n}_{\textn{f}}$. Let $t^{0}_{\textn{f}}\eq0$, and $t_{\textn{f}}^{n}\eq\sum_{j=1}^{n}\tau^{j}_{\textn{f}}, \ 1 \leq n\leq N_{\textn{f}}$ be the discrete times for the flow problem. Let
$I^n_{\textn{f}}\eq(t^{n-1}_{\textn{f}},t^n_{\textn{f}}], \ 1 \le n \le N_{\textn{f}}$.   
 For the time stepping for the problem of mechanics, we will restrict ourselves to the case in which  a fixed number  of 
 local flow time steps  corresponds to one coarse mechanics time step. We    suppose that  $N_{\textn{f}}=\delta_{\textn{fm}} N_{\textn{m}} $,  
 with $\delta_{\textn{fm}}$ and $N_{\textn{m}}$ are given positive integer values, 
where  $\delta_{\textn{fm}}$ is   the fixed number  of local flow time steps   within one coarse mechanics time step. 
We then  let $\left(\tau^{\ell}_{\textn{m}}\right)_{1\leq \ell\leq N_{\textn{m}}}$ 
such that  $T=\sum_{\ell=1}^{ N_{\textn{m}}}\tau^{\ell}_{\textn{m}}$; 
for $1\leq \ell\leq N_{\textn{m}}$. We have then   
$\tau^{\ell}_{\textn{m}}=\sum_{n=(\ell-1)\delta_{\textn{fm}}+1}^{\ell\delta_{\textn{fm}}}\tau^{n}_{\textn{f}}$, and we let 
$t^{0}_{\textn{m}}\eq0$, and $t_{\textn{m}}^{\ell}\eq\sum_{j=1}^{\ell}\tau^{\ell}_{\textn{m}}, \ 1 \leq \ell\leq  N_{\textn{m}}$ be 
the discrete times for the problem of mechanics; see~\cite{almani2016convergence} for a similar notation. We use $I^{\ell}_{\textn{m}}\eq(t^{\ell-1}_{\textn{m}},t^{\ell}_{\textn{m}}], \ 1 \le \ell \le  N_{\textn{m}}$.
For any sufficiently smooth function $v_{\htau}$, 
 we use the notation  $v_{h}^{n} := v_{\htau}(\cdot,t^{n}_{\textn{f}})$, for all $ 0\leq n \leq N_{\textn{f}}$.


%
\subsection{Partition of the domain $\Omega$}
 Let  $\mathcal{T}_h$ be a simplicial  mesh of $\Omega$,  
matching in the sense that for two distinct elements of $\mathcal{T}_h$ their
intersection is either an empty set or their common vertex or edge.  Let $h_{K}$ denote the diameter of $K\in \Tau_{h}$ and $h$
be the largest diameter of all triangle; $h\eq \max_{K\in \Tau_{h}}h_K$. The set of 
vertices of the mesh $\Tau_{h}$ is denoted by $\Vauh$, $\Vauh^{\textn{int}}$
for the set of interior vertices, and $\Vauh^{\textn{ext}}$ for the set of
boundary vertices.  For each $\veca\in \Vauh$, let $\Tau_{h}^{\veca}$
denote the patch of the vertex $\veca$, i.e., all the elements
 $K\in \Tau_{h}$ which share the vertex $\veca$. We denote by $\omega_{\veca}$
 the corresponding open subset of $\Omega$.

\subsection{Discrete function spaces}

To approximate the flow subproblem~\eqref{fstress_iterative_flow_optimized}, we let $Q_{h}\times \mathbf{W}_{h}\subset Q\times \mathbf{W}$ be 
the Raviart--Thomas--N\'{e}d\'{e}lec mixed finite element spaces of order zero on the mesh $\Tau_{h}$~(cf.~\cite{arnold1985mixed}):
\bse\label{flowspaces}\begin{alignat*}{2}
&Q_{h}:=\displaystyle\{q_{h}\in L^{2}(\Omega);\,\forall K\in \Tau_h,\,q_{h}|_{K}\in \amsmathbb{P}_{0}(K) \},\\
&\mathbf{W}_{h}:=\displaystyle\{ \vecv_{h}\in \Hdiv;\,\forall K\in \Tau_h,\,\vecv_{h}|_{K}\in \RTN_0(K)\},
\end{alignat*}\ese
where $\RTN_0(K)$ denotes the
lowest-order Raviart--Thomas--N\'{e}d\'{e}lec finite-dimensional subspace associated with the element $K\in\Tau_{h}$. 
%

To approximate  the  mechanics subproblem~\eqref{fstress_iterative_mechanics_optimized},  
we let $\mathbf{Q}_{h}\times \amsmathbb{W}_{h}\times\amsmathbb{Q}_{\textn{sk},h}\subset \mathbf{Q}\times \amsmathbb{W}\times\amsmathbb{Q}_{\textn{sk}}$ be the Arnold--Falk--Winther 
mixed finite elements with weakly symmetric stress for the lowest-order stresses on the mesh $\Tau_{h}$~(cf.~\cite{arnold2007mixed}):
\bse\label{mechspaces}\begin{alignat*}{2}
&\mathbf{Q}_{h}:=\displaystyle\{\vecz_{h}\in \mathbf{L}^{2}(\Omega);\,\forall K\in \Tau_h,\,\vecz_{h}|_{K}\in \left[\amsmathbb{P}_{0}(K)\right]^{d} \},\\
&\amsmathbb{W}_{h}:=\displaystyle\{ \vectau_{h}\in \Htensordiv;\,\forall K\in \Tau_h,\,\vectau_{h}|_{K}\in \left[\amsmathbb{P}_{1}(K)\right]^{d\times d}\},\\
&\amsmathbb{Q}_{\textn{sk},h}:=\displaystyle\{\gamma_{h}\in [L^{2}(\Omega)]_{\textn{sk}}^{d\times d};\,\forall K\in \Tau_h,\,\gamma_{h}|_{K}\in [\amsmathbb{P}_{0}(K)]_{\textn{sk}}^{d\times d} \},
\end{alignat*}\ese
where   $[\amsmathbb{P}_{0}(K)]_{\textn{sk}}^{d\times d}$ denotes  the  subspace  of $[\amsmathbb{P}_{0}(K)]^{d\times d}$ 
composed  of   skew symmetric--valued tensors. 
 
Let $E$ be a space of functions defined on $\Omega$. We denote
$P^{1}_{\tau}(E)$
the vector space of functions continuous  in
time and with values in $E$. We also denote
by $P^{0}_{\tau}(E)$ the space of functions piecewise constant in
time and with values in $E$. We have then if $v_{\htau}\in P^{1}_{\tau}(E)$, then $\partial_{t}v_{\htau}\in P^{0}_{\tau}(E)$ is such that  for all $1\leq n\leq N_{\textn{f}}$,
\begin{alignat}{2}\label{dtnfunc}
&\partial_{t} v^{n}_{h} := \partial_{t} v_{\htau}|_{I_{\textn{f}}^{n}} = \dfrac{v_{h}^{n}-v_{h}^{n-1}}{\tau^{n}_{\textn{f}}}.
\end{alignat}

 \section{Fully discrete space-time fixed-stress schemes based on MFE in-space
and the backward Euler scheme in-time}\label{sec:standard_discrete_algorithms}
In this section, we provide two discretization of Algorithm~\eqref{fstress_iterative_flow_optimized}-\eqref{fstress_iterative_mechanics_optimized} using  
the backward Euler scheme in-time, and in-space,  using 
two mixed finite elements methods  for the linear elasticity and  flow problems.  A post-processing routine is 
then given to preview the numerical pressure and displacement solutions.
\subsection{Two  standard discrete fixed-stress schemes.}\label{subsect:discrete_MFE_solution}
In the first algorithm, we   consider the case  of equal time grids for the flow and mechanics problems, 
i.e., $N_{\textn{f}} =N_{\textn{m}}$. The fully discrete  form of Algorithm~\ref{original_fs} reads  then as follows:
\begin{algo}[The global-in-time fixed-stress]~\label{space_time_discrete}
{
\setlist[enumerate]{topsep=0pt,itemsep=-1ex,partopsep=1ex,parsep=1ex,leftmargin=1.5\parindent,font=\upshape}
\begin{enumerate}
\item Chose an initial approximation $\vecsigma^{0}_{\htau}\in P^{0}_{\tau}(\mathbf{W}_{h})$ of 
$\vecsigma$, a real constant $\beta>0$,  and a tolerance $\epsilon>0$. 
Set $k\eq-1$.
\item \textn{\textbf{Do}}
\begin{enumerate}
\item  Increase $k\eq k+1$ and set $n\eq0$.
\item \textn{\textbf{Do}}
\begin{enumerate}
\item Increase $n\eq n+1$.
\item Approximate   $(\vecw_{h}^{k,n},p^{k,n}_{h})\in \mathbf{W}_{h}\times Q_{h}$,  the solution to 
\bse\label{fstress_iterative_flow_std}\begin{alignat}{4}
&(\vK^{-1}\vecw^{k,n}_{h},\vecv)-(p^{k,n}_{h},\divecv\vecv)=0,&&\quad\forall \vecv\in \mathbf{W}_{h}.\label{fstress_iterative_flow_darcy_std}\\
\nonumber&(c_{0}+c_{r}+\beta)(\partial_{t} p_{h}^{k,n},q)+(\divecv\vecw_{h}^{k,n},q)= (g^{n},q)&&\\
&\qquad\qquad\qquad +\beta(\partial_{t} p_{h}^{k-1,n},q)
-\dfrac{c_{r}}{d\alpha}(\partial_{t} \vecsigma_{h}^{k-1,n},q\II ),&&\quad\forall q\in Q_{h},\label{fstress_iterative_flow_cons_std}
\end{alignat}\ese
 \end{enumerate}
 \textn{\textbf{While}} $n\leq N_{\textn{f}}$.
\item Reset $n\eq0$.
\item \textn{\textbf{Do}}
\begin{enumerate}
\item Increase $n\eq n+1$.
\item Approximate $(\vecsigma_{h}^{k,n},\vecu_{h}^{k,n},\veczeta_{h}^{k,n})\in \amsmathbb{W}_{h}\times \mathbf{Q}_{h}\times 
\amsmathbb{Q}_{\textn{sk},h}$, solution to
\bse\label{fstress_iterative_mech_std}\begin{alignat}{4}
&(\mathcal{A}\vecsigma_{h}^{k,n},\vectau)+(\vecu_{h}^{k,n},\divecv\vectau)+(\veczeta_{h}^{k,n},\vectau)=-\dfrac{c_{r}}{d\alpha}( p_{h}^{k,n}\II,\vectau),&&\quad\forall \vectau\in \amsmathbb{W}_{h},\label{fstress_iterative_mech_stress_std}\\
&(\divecv\vecsigma_{h}^{k,n},\vecz)+(\vecsigma_{h}^{k,n},\vecgamma)=-(\vf^{n},\vecz),&&\quad\forall (\vecz,\vecgamma)\in \mathbf{Q}_{h}\times \amsmathbb{Q}_{\textn{sk},h},\label{fstress_iterative_cons_std}
\end{alignat}\ese
\end{enumerate}
\textn{\textbf{While}} $n\leq N_{\textn{f}}$.
\end{enumerate}
 \textn{\textbf{While}}  $\left(\dfrac{\sum_{n=1}^{N_{\textn{f}}}\|(\vecsigma^{k,n}_{h},p^{k,n}_{h}) - (\vecsigma^{k-1,n}_{h},p^{k-1,n}_{h})\|^{2}}{\sum_{n=1}^{N_{\textn{f}}}\| (\vecsigma^{k-1,n}_{h},p^{k-1,n}_{h})\|^{2}}\right)^{\frac{1}{2}}\geq \epsilon$.\hfill\eqnum\label{class_stp_critera1}
 \end{enumerate}}
\end{algo}

We present  now  the  nonconforming-in-time counterpart of Algorithm~\ref{original_fs} in the spirit of 
multi-rate fixed-stress scheme specified in~\cite{almani2016convergence}:
\begin{algo}[The nonconforming-in-time (multi-rate) fixed-stress]~\label{multi-rate_discrete} 
{
\setlist[enumerate]{topsep=0pt,itemsep=-1ex,partopsep=1ex,parsep=1ex,leftmargin=1.5\parindent,font=\upshape}
\begin{enumerate}
 \item  Chose an initial approximation $\vecsigma^{0}_{\htau}\in P^{0}_{\tau}(\mathbf{W}_{h})$ of 
$\vecsigma$, a real constant $\beta>0$,  and a tolerance $\epsilon>0$. 
Set $\ell\eq-\delta_{\textn{fm}}$.
\item \textn{\textbf{Do}}
\begin{enumerate}
\item Increase $\ell\eq \ell+\delta_{\textn{fm}}$ and set $k\eq-1$.
\item \textn{\textbf{Do}}
\begin{enumerate}
\item Increase $k\eq k+1$ and set $m\eq0$.
\item \textn{\textbf{Do}}
\begin{enumerate}
\item Increase $m\eq m+1$.
\item Approximate   $(\vecw_{h}^{k,\ell+m},p^{k,\ell+m}_{h})\in \mathbf{W}_{h} \times Q_{h}$, solution to
    \bse\label{fstress_iterative_flow_multi-rate}\begin{alignat}{4}
&(\vK^{-1}\vecw^{k,\ell+m}_{h},\vecv)-(p^{k,\ell+m}_{h},\divecv\vecv)=0,&&\quad\forall \vecv\in \mathbf{W}_{h}.\label{fstress_iterative_flow_darcy_multi-rate}\\
\nonumber&(c_{0}+c_{r}+\beta)(\partial_{t} p_{h}^{k,\ell+m},q)+(\divecv\vecw_{h}^{k,\ell+m},q)= (g^{\ell+m},q)&&\\
&\qquad\qquad +\beta(\partial_{t}p_{h}^{k-1,\ell+m},q)
-\dfrac{c_{r}}{d\alpha}(\dfrac{\vecsigma_{h}^{k-1,\ell+\delta_{\textn{fm}}}-\vecsigma_{h}^{k-1,\ell}}{\tau^{\ell+\delta_{\textn{fm}}}_{\textn{m}}},q\II),&&\quad\forall q\in Q_{h},\label{fstress_iterative_flow_cons_multi-rate}
\end{alignat}\ese
\end{enumerate}
\textn{\textbf{While}} $m\leq \delta_{\textn{fm}}$.
\item Approximate $(\vecsigma_{h}^{k,\ell+\delta_{\textn{fm}}},\vecu_{h}^{k,\ell+\delta_{\textn{fm}}},\veczeta_{h}^{k,\ell+\delta_{\textn{fm}}})\in \amsmathbb{W}_{h}\times \mathbf{Q}_{h}\times 
\amsmathbb{Q}_{\textn{sk},h}$, solution to
\bse\label{fstress_iterative_mech_multi-rate}\begin{alignat}{4}
&(\mathcal{A}\vecsigma_{h}^{k,\ell+\delta_{\textn{fm}}},\vectau)+(\vecu_{h}^{k,\ell+\delta_{\textn{fm}}},\divecv\vectau)+(\veczeta_{h}^{k,\ell+\delta_{\textn{fm}}},\vectau)=-\dfrac{c_{r}}{d\alpha}(p_{h}^{k,\ell+\delta_{\textn{fm}}} \II ,\vectau),\quad\forall \vectau\in \amsmathbb{W}_{h},&&\label{fstress_iterative_mech_stress_multi-rate}\\
&(\divecv\vecsigma_{h}^{k,\ell+\delta_{\textn{fm}}},\vecz)+(\vecsigma_{h}^{k,\ell+\delta_{\textn{fm}}},\vecgamma)=-(\vf^{\ell+\delta_{\textn{fm}}},\vecz),\quad\forall (\vecz,\vecgamma)\in \mathbf{Q}_{h}\times \amsmathbb{Q}_{\textn{sk},h},\label{fstress_iterative_cons_multi-rate}
\end{alignat}\ese 
\end{enumerate}
 \textn{\textbf{While}}   $\left(\dfrac{\|\vecsigma^{k,\ell+\delta_{\textn{fm}}}_{h} - \vecsigma^{k-1,\ell+\delta_{\textn{fm}}}_{h}\|^{2}+\sum_{m=1}^{\delta_{\textn{fm}}}\|p^{k,\ell+m}_{h} - p^{k-1,\ell+m}_{h}\|^{2}}{\| \vecsigma^{k-1,\ell+\delta_{\textn{fm}}}_{h}\|^{2}+\sum_{m=1}^{\delta_{\textn{fm}}}\| p^{k-1,\ell+m}_{h}\|^{2}}\right)^{\frac{1}{2}}\geq \epsilon$.\hfill\eqnum\label{class_stp_critera2}
 \end{enumerate}
\textn{\textbf{While}} $\ell <\delta_{\textn{fm}} N_{\textn{m}}$.
\end{enumerate}
}
\end{algo}
\begin{rem}[The multi-rate FS]
The convergence of the multi-rate fixed-stress was 
shown in~\cite{almani2016convergence}  where  mixed finite element method is  used 
for the flow equations and where the mechanics is solved
by conformal Galerkin method. Therein, the algorithm is also limited to one coarser time step for the mechanics 
and there is no study on  the propagation of error due to 
temporal and spatial discretizations.
\end{rem}
\begin{rem}[Space-time  vs multi-rate]
We first notice that Algorithm~\ref{multi-rate_discrete}  is   practical 
to problems  with a long time integration interval; In contrast  to Algorithm~\ref{space_time_discrete}, it  
requires reasonable computation ability and less storage resources to handle large-scale applications. 
Furthermore,  Algorithm~\ref{multi-rate_discrete} is exploiting the different time scales for 
the  flow and mechanics subsystems. 
We note also that
the efficiency of the two algorithms can be   improved  when the free parameter $\beta$ is  well-chosen (see~\cite{storvik2018optimization}) and   that 
step 2.(d) of Algorithm~\ref{space_time_discrete}  in practice,  
is  done in parallel as in~\cite{borregales2018parallel}.
\end{rem}

\subsection{Post-processing} \label{sec:post-processing}
We do here some improvements to   the    approximate  solution 
$(p^{k}_{\htau},\vecu^{k}_{\htau})$~(cf.~\cite{ern2010posteriori,lee2016robust}). This step is also
mandatory  to  design from Algorithm~\ref{space_time_discrete} and~\ref{multi-rate_discrete},  their
adaptive   versions based on  energy-norm-type a posteriori error estimate. This
is customary in  mixed finite elements schemes, 
as an energy-norm-type a posteriori error estimate has no meaning for the piecewise constant, i.e., 
$\nabla p^{k,n}_{h}=\nabla \vecu^{k,n}_{h}=0$. 

Let us  notice first that  in~Algorithm~\ref{multi-rate_discrete}, the approximate solution  
$(\vecu^{k}_{\htau},\sigmabo^{k}_{\htau})$ of the mechanics problem is defined 
in   different time grids from the    approximate flow  solution 
$(p^{k}_{\htau},\vecw^{k}_{\htau})$, so  we cannot proceed to the post-processing of 
the displacement and the reconstruction of  the stress tensor unless we build the couple 
$(\vecu^{k}_{\htau}, \sigmabo^{k}_{\htau})$ at the finer time steps $t^{n}_{\textn{f}}=t^{\ell+m}_{\textn{f}}$, 
for all $\ell=0,\delta_{\textn{fm}},2\delta_{\textn{fm}},3\delta_{\textn{fm}},\cdots,(N_{\textn{m}}-1)\delta_{\textn{fm}}$, and for all $1\leq m\leq \delta_{\textn{fm}}-1$. 
To this aim, we  construct
the displacement  and the stress tensor as follows: for $\ell=0,\delta_{\textn{fm}},2\delta_{\textn{fm}},3\delta_{\textn{fm}},\cdots,(N_{\textn{m}}-1)\delta_{\textn{fm}}$, we set
\bse\label{potsprossessing_displastress_finestep}\begin{alignat}{2}
 \vecu^{k,\ell+m}_{h}&\eq\vecu^{k,\ell}_{h}+\dfrac{m}{\delta_{\textn{fm}}}\vecu^{k,\ell+\delta_{\textn{fm}}}_{h},\qquad 1\leq  m \leq \delta_{\textn{fm}}-1,\\
 \vecsigma^{k,\ell+m}_{h}&\eq\vecsigma^{k,\ell}_{h}+\dfrac{m}{\delta_{\textn{fm}}}\vecsigma^{k,\ell+\delta_{\textn{fm}}}_{h},\qquad 1 \leq m \leq \delta_{\textn{fm}}-1.
\end{alignat}\ese
Note that this post-processing is explicit and its cost is negligible. 
The post-processing   of the pressure  $p^{k,n}_{h}$ is as follows~\cite{ern2010posteriori,lee2016robust,HannukainenSV12}:  
at each iteration $k\geq 1$,  we 
calculate the improved solution 
$\widetilde{p}^{k,n}_{h}\in\amsmathbb{P}_{2}(\Tau_{h})$ in each element $K\in \Tau_{h}$ such that
\bse\label{potsprossessing_pressure}\begin{alignat}{2}
\label{potsprossessing_pressure_eq} -\vK\nabla \widetilde{p}^{k,n}_{h}&=\vecw^{k,n}_{h},&&\quad\forall K\in \Tau_{h},\\
\label{potsprossessing_pressure_mean} (\widetilde{p}^{k,n}_{h},1)_{K}&=(p^{k,n}_{h},1)_{K},&&\quad\forall K\in \Tau_{h}.
\end{alignat}\ese
This  post-processing is computationally cheap and easy to be implemented.  We then  extend this  post-processing 
(cf. \cite{lee2016robust}) to the vector case, leading to 
define a function   $\widetilde{\vecu}^{k,n}_{h}\in\left[\amsmathbb{P}_{2}(\Tau_{h})\right]^{d}$, such that
\bse\label{potsprossessing_displacement}\begin{alignat}{2}
\label{potsprossessing_displacement_eq}  \nabla\widetilde{\vecu}^{k,n}_{h}-\dfrac{c_{r}}{d\alpha} p^{k,n}_{h}\II-\veczeta^{k,n}_{h}&=\mathcal{A}\vecsigma^{k,n}_{h},&&\quad\forall K\in \Tau_{h},\\
\label{potsprossessing_displacement_mean} \dfrac{(\widetilde{\vecu}^{k,n}_{h},\mathbf{e}_{i})_{K}}{|K|}&=\vecu^{i,k,n}_{h}|_{K},&&\quad i=1,\cdots,d,\,\,\forall K\in \Tau_{h},
\end{alignat}\ese
where  $\mathbf{e}_{i}\in\RR^{d}$ denotes  the $i$-th Euclidean unit vector.
 We finally  assume for simplicity that the initial conditions  are  satisfied  exactly,  i.e., 
 \begin{equation}\label{initial_data_satisfied}
\widetilde{p}^{k}_{\htau}(\cdot,0)= p_{0},\quad\textn{ and }\quad\widetilde{\vecu}^{k}_{\htau}(\cdot,0)= \vecu_{0}.  
 \end{equation}
  We define the continuous and piecewise affine in--time functions $\widetilde{p}^{k}_{\htau}$ and $\widetilde{\vecu}^{k}_{\htau}$ by
\begin{equation}\label{postproc_continuous_intime}
 \widetilde{p}^{k}_{\htau}(\cdot,t^{n}_{\textnormal{f}})=\widetilde{p}^{k,n}_{h},\quad \widetilde{\vecu}^{k}_{\htau}(\cdot,t^{n}_{\textnormal{f}})=\widetilde{\vecu}^{k,n}_{h},\quad 0\leq n\leq N_{\textnormal{f}}.
 \end{equation}
The key observation in the above post-processing is that they  use only local
operations, which are independent from each other
and hence parallelizable.
\section{The adaptive fixed-stress algorithms}
\label{sec:adaptive_discrete_algorithms}
The purpose of this section  is to  reduce as much as possible   
the computational effort of Algorithm~\ref{space_time_discrete} and~\ref{multi-rate_discrete} 
as in~\cite{
ahmed:hal-01540956,hassan2017posteriori, di2015adaptive,adaptnewernvohralik}. 
The improvements of these two standards algorithms stems from (i) {important savings in terms 
of the number of  coupling iterations can be achieved using adaptive stopping criterion} (ii)
{a significant gain in the computational resources is obtained  by balancing the error components via an asynchronous adaptivity of the  temporal meshes} (iii) 
{optimizing the tuning parameter $\beta$}. 
\subsection{Methodology for  adaptive  asynchronous  time-stepping and adaptive stopping criteria}\label{sec:balanc_stopping}
Let $\eta^{k,n}_{\textn{sp,J}}$, $\eta^{k,n}_{\textn{tm,J}}$ and 
$\eta^{k,n}_{\textn{it,J}}$, for $\textn{J}=\textn{P,\,U}$,  be respectively the estimators of 
\textit{the spatial discretization error}, 
\textit{the temporal discretization error} and \textit{the fixed-stress coupling error} at the 
$n\textn{-th}$ time step and on the $k\textn{-th}$ iteration, where the index
$\textn{J=P}$ is for \textit{the pressure error components}, and  that of $\textn{J=U}$ is for   
\textit{the displacement error components}. We let $\eta^{k,n}_{\bullet}
\eq\eta^{k,n}_{\bullet,\textn{P}} +\eta^{k,n}_{\bullet,\textn{U}}$.

The first step of our developments  is to equip~Algorithm~\ref{space_time_discrete} and~\ref{multi-rate_discrete}  
with  adaptive asynchronous time-stepping.  To this aim, we propose 
to equilibrate the time errors with the spatial errors as follows; we adjust the  time steps 
$\tau^{n}_{\textn{f}}$ and $\tau^{n}_{\textn{m}}$ so that
\begin{equation}\label{space_time_balance}
 \gamma_{\textn{tm,J}}\eta^{k,n}_{\textn{sp,J}}\leq \eta^{k,n}_{\textn{tm,J}}\leq \Gamma_{\textn{tm,J}} \eta^{k,n}_{\textn{sp,J}},\quad\textn{J=P,\,U},
\end{equation}
where $\gamma_{\textn{tm,J}}$ and $\Gamma_{\textn{tm,J}}$, $\textn{J=P,\,U},$ are  user-given weights, typically close to $1$.
{
An alternative to~\eqref{space_time_balance} being  to balance the time errors from the flow  and   mechanics 
discretization with the global error
by selecting the time steps $\tau^{n}_{\textn{m}}$ and $\tau^{n}_{\textn{f}}$  in such a way that
 \begin{equation}\label{flow__mech_time_balance}
  \gamma_{\textn{tm,J}}\max(\eta^{k,n}_{\textn{sp,U}},\eta^{k,n}_{\textn{sp,P}})\leq \eta^{k,n}_{\textn{tm,J}}\leq \Gamma_{\textn{tm,U}} \max(\eta^{k,n}_{\textn{sp,U}},\eta^{k,n}_{\textn{sp,P}}),\quad\textn{J=P,\,U}.
 \end{equation}  }
 {
The  balancing criterion~\eqref{space_time_balance}  controls  
 the contributions of $\eta^{k,n}_{\textn{tm,P}}$ and 
 $\eta^{k,n}_{\textn{tm,U}}$ in the
 overall error and leads to 
 $\eta^{k,n}_{\textn{sp}}\approx\eta^{k,n}_{\textn{tm}}$. That of the second criterion~\eqref{flow__mech_time_balance}   leads  to equilibrate the time errors from the flow and mechanics, i.e., 
 $\eta^{k,n}_{\textn{tm,P}} \approx\eta^{k,n}_{\textn{tm,U}}$.}
 

The second step of our developments is to equip Algorithm~\ref{space_time_discrete} and~\ref{multi-rate_discrete}
with  adaptive stopping criteria. We then
introduce a real parameter $\gamma_{\textn{it}}$ to  be
given in $(0,1)$. The stopping criteria for Algorithm~\ref{space_time_discrete}
 is  chosen at each iteration $k$ as
\begin{equation}
  \label{Criteria_space_time}
  \eta^{k}_{\textn{it}}\leq \gamma_{\textn{it}} \max\Big\{\eta^{k}_{\textn{sp}}
  ,\eta^{k}_{\textn{tm}}\Big\},
\end{equation}
which implies that we  stop the iterations  if the coupling error is
sufficiently lower than  one of the other components. The stopping criteria for 
Algorithm~\ref{multi-rate_discrete}  is set similarly: 
at each coarse  mechanics time step  $\ell=0,\delta_{\textn{fm}},2\delta_{\textn{fm}},\cdots, (N_{\textn{m}}-1)\delta_{\textn{fm}}$,  at each iteration $k\geq 1$, 
\begin{equation}
  \label{Criteria_multi-rate}
  \sum_{m=1}^{\delta_{\textn{fm}}}\eta^{k,\ell+m}_{\textn{it}}\leq \gamma_{\textn{it}} \max\Big\{\sum_{m=1}^{\delta_{\textn{fm}}}\eta^{k,\ell+m}_{\textn{sp}}
  ,\sum_{m=1}^{\delta_{\textn{fm}}}\eta^{k,\ell+m}_{\textn{tm}}\Big\}.
\end{equation}
\begin{rem}[Algebraic errors]
 The  systems within the flow and  mechanics subsystems are solved with 
 direct solvers.  The present adaptive approach can also be combined with an iterative solver for each subproblem, 
 and to further save computational effort, these latter can be stopped whenever the algebraic errors does not contribute 
 significantly to the overall error, following~\cite{jiranek2010posteriori}. 
\end{rem}
\subsection{The adaptive algorithms}
We are now ready to present the adaptive counterparts of Algorithm~\ref{space_time_discrete} and~\ref{multi-rate_discrete}, i.e., 
they are now equipped  with adaptive asynchronous 
time-stepping and a posteriori stopping criterion. 
The adaptive version of Algorithm~\ref{space_time_discrete} 
is as follows:
\begin{algo}[Adaptive  Fixed-Stress with Asynchronous Time Mesh Refinement and Adaptive Stopping Criteria ]~\label{space_time_discrete_adaptive}
{
\setlist[enumerate]{topsep=0pt,itemsep=-1ex,partopsep=1ex,parsep=1ex,leftmargin=1.5\parindent,font=\upshape}
\begin{enumerate}
\item In step~1, chose also a real  parameter $\gamma_{\textn{it}}$ and the 
real  weights $\gamma_{\textn{tm}}$ and $\Gamma_{\textn{tm}}$, set $t_{\textn{f}}^{0}=0$, and give an initial ratio $\delta_{\textn{fm}}^{0}$ and time step for the flow $\tau^{0}_{\textn{f}}$ 
and the temporal refinement threshold $\tau_{\textn{min}}$. Set $k\eq-1$.
\item \textn{\textbf{Do}}
\begin{enumerate}
\item Increase $k\eq k+1$ and set $n\eq0$.
\item \textn{\textbf{Do}}
\begin{enumerate}
\item Increase $n\eq n+1$.
\item Set $\tau^{n}_{\textn{f}}\eq \tau^{n-1}_{\textn{f}}$.
\item Approximate   $(\vecw_{h}^{k,n},p^{k,n}_{h})$ by~\eqref{fstress_iterative_flow_std}.
\item Calculate the estimators $\eta^{k,n}_{\textn{sp,P}}$, 
$\eta^{k,n}_{\textn{tm,P}}$ and $\eta^{k,n}_{\textn{it,P}}$.~Check the balancing criterion~\eqref{space_time_balance} \textn{(}or~\eqref{flow__mech_time_balance}\textn{)}.~If not satisfied, refine or redefine the flow time step $\tau^{n}_{\textn{f}}$ in such 
a way that condition~\eqref{space_time_balance} \textn{(}or~\eqref{flow__mech_time_balance}\textn{)} holds or $\tau^{n}_{\textn{f}}\leq \tau_{\textn{min}}$, and return to step~\textn{2.(b)iii}.
\item Set $t ^{n}_{\textn{f}}\eq t^{n-1}_{\textn{f}}+\tau^{n}_{\textn{f}}$.
\end{enumerate}
 \textn{\textbf{While}} $t_{\textn{f}}^{n}\leq T$.
\item Reset $n\eq0$ and  $t_{\textn{m}}^{0}=0$, and let 
 $\tau^{0}_{\textn{m}}=\delta_{\textn{fm}}^{0}\tau^{0}_{\textn{f}}$.
\item \textn{\textbf{Do}}
\begin{enumerate}
 \item Increase $n\eq n+1$.
\item Set $\tau^{n}_{\textn{m}}\eq \tau^{n-1}_{\textn{m}}$.
\item Approximate $(\vecsigma_{h}^{k,n},\vecu_{h}^{k,n},\veczeta_{h}^{k,n})$ by~\eqref{fstress_iterative_mech_std}.
\item Calculate the estimators $\eta^{k,n}_{\textn{sp,U}}$, 
$\eta^{k,n}_{\textn{tm,U}}$ and $\eta^{k,n}_{\textn{it,U}}$.~Check the balancing criterion~\eqref{space_time_balance} \textn{(}or~\eqref{flow__mech_time_balance}\textn{)}.~If not satisfied, redefine the  time step for the mechanics $\tau^{n}_{\textn{m}}$ using 
using $\left(\tau^{n}_{\textn{f}}\right)_{n\geq 0}$ in such a way that condition~\eqref{space_time_balance} \textn{(}or~\eqref{flow__mech_time_balance}\textn{)} holds, and return to step~\textn{2.(d)iii}.
\item Set $t ^{n}_{\textn{m}}\eq t^{n-1}_{\textn{m}}+\tau^{n}_{\textn{m}}$.
\end{enumerate}
\textn{\textbf{While}} $t_{\textn{m}}^{n}\leq T$.
\end{enumerate}
\textn{\textbf{Until}} the criteria~\eqref{Criteria_space_time}  is satisfied.
\end{enumerate}
}
\end{algo}
Similarly, we propose to modify  Algorithm~\ref{multi-rate_discrete}. This yields to  an 
adaptive fixed-stress scheme equivalent to Algorithm~\ref{space_time_discrete_adaptive} 
but applied through \textit{temporal windowing technique}. Precisely, the whole time interval $[0,T]$ is now split into
$N$ time-windows $[0,T_{1} ]$, $[T_{1},T_{2}],\cdots,[T_{N-1},T]$.
Algorithm~\ref{space_time_discrete_adaptive} is first applied on the first
time window $[0,T_{1} ]$. Afterwards, one applies the algorithm on
the next time window $[T_{i-1},T_{i} ]$
imposing as initial condition for $t^{0}_{\textn{f}}= t^{0}_{\textn{m}}=T_{i-1}$
the solution of the converged iterate of the end of the previous time window and 
proceeds in such a way until all time windows have been treated. 
\begin{rem}[Space  adaptivity]
The  estimators are calculated on each element of the mesh
and on each time step, and could also be used as indicators in order to
refine adaptively the   spatial mesh~$\Tau_{h}$,  so that the local spatial error 
estimators are distributed equally;  see~\cite{AHMED2018, di2015adaptive,Vohralik_Wheeler2013} and
the references therein.  Furthemore, the  efficiency of the adaptive algorithms can be enhanced by using adaptive multiscale meshes, where 
two spatial meshes  for the flow and mechanics subsystems are considered   and where they are 
refined/coarsened adaptively  in order  to equilibrate the space  errors for the two  subsystems; see the 
multiscale discretizations techniques  in~\cite{DANA20181,pencheva2013robust}.
\end{rem}
\begin{rem}[Static condensation]
 The decoupling procedure permits  the use of a local static condensation for the flow and mechanics  and then   to reduce the MFE system resulting from each subproblem to a symmetric and positive definite one;    for  the  pressure for the flow problem, and for the  displacement and rotation for the mechanics problem;   with the same way as in \cite{ambartsumyan2018multipoint1,ambartsumyan2018multipoint2}.  These systems are smaller and easier to solve than the original saddle point problems, but no further reduction is possible. 
\end{rem}

 \begin{rem}[Computational cost]
 In practice,  the adaptive time-stepping  strategy  is only done in  the first or second iterations of the Algorithm.   Further, when the  step size is modified  once at current time step, the updated step size can be  used for  some subsequent time steps, say for example $5$ time steps. Then  on the sixth time step,  step~2.(b)iv
is checked again  if the step size needs to be modified for the next $5$ time  steps. Furthermore,  steps~2.(b)iv and~2.(d)iv can be  done only every few iterations of the Algorithm. 
 \end{rem}

\section{Concept of $H^1$-, and  $\mathbf{H}(\textnormal{div})$-conforming reconstructions in Biot's poro-elasticity system } \label{sec:reconstructions}
In this section, we develop basic tools that will allow us to build the estimators involved in the adaptive 
fixed-stress algorithms.
\subsection{Pressure and displacement  reconstructions}
We first  construct from $\widetilde{p}^{k,n}_{h}$,  a $H^{1}_{0}$-conforming function $\hat{p}^{k,n}_{h}$ 
satisfying the  mean  value constraint
\begin{equation}\label{rec_pressure_mean}
 \left(\hat{p}^{k,n}_{h},1\right)=\left(\widetilde{p}^{k,n}_{h},1\right)_{K},\qquad  \forall K\in \Tau_{h}.
\end{equation}
To this aim, we proceed as in~\cite{ern2010posteriori}; from the available post-processed pressure $\widetilde{p}^{k,n}_{h}$
at each iteration $k\geq 1$,   we set
\[  {\hat{p}}^{k,n}_{h}(\veca) \eq \Iavo(\widetilde{p}^{k,n}_{h})(\veca)+\sum_{K\in\Tau_{h}}a_{K}^{k,n}b_{K}(\veca).
\]
 Here $\veca$ are the  Lagrangian nodes  situated in the interior of $\Omega$, 
 $b_{K}$ denotes the standard (time-independent) bubble function supported on $K$, 
 defined as the product of the barycentric coordinates
of $K$, for all $K\in\Tau_{h}$, and scaled so that its maximal value is 1, and
$\Iavo : \bP_2(\Tau_{h}) \rightarrow \bP_2(\Tau_{h}) \cap
H^1(\Omega)$ is the interpolation operator given by
$$
\Iavo(\phi_h)(\veca)=\frac{1}{|\Tau_{\veca}|} \sum_{K \in \Tau^{\veca}_{h}} \phi_h|_K(\veca).
$$
At the Lagrange nodes $\veca$ situated on the boundary
$\partial\Omega$, we set ${\hat{p}}^{k,n}_{h}(\veca) \eq 0$. In order to guarantee that the 
mean value constraint~\eqref{rec_pressure_mean} holds true, we choose
\begin{equation}
 a^{k,n}_{K}=\dfrac{1}{(b_{K},1)_{K}}\left(\widetilde{p}^{k,n}_{h}-\Iavo(\widetilde{p}^{k,n}_{h}),1\right)_{K}.
\end{equation}
The same procedure can be applied to the post-processed displacement 
$\widetilde{\vecu}^{k,n}_{h}\in\left[\amsmathbb{P}_{2}(\Tau_{h})\right]^{d}$. This  
leads to a $\mathbf{H}^{1}_{0}(\Omega)$-conforming vector function $\hat{\vecu}^{k,n}_{h}$, satisfying the 
following mean  value constraint:
\begin{equation}\label{rec_displacement_mean}
\qquad\qquad\qquad\left(\hat{\vecu}^{k,n}_{h},\mathbf{e}_{i}\right)=\left(\widetilde{\vecu}^{k,n}_{h},\mathbf{e}_{i}\right)_{K},\qquad i=1,\cdots,d,\,\, \forall K\in \Tau_{h}.
\end{equation}
We end up with  continuous  and piecewise affine in--time functions $\widetilde{p}^{k}_{\htau}$ and $\widetilde{\vecu}^{k}_{\htau}$ by setting
\begin{equation}\label{rec_continuous_intime}
 \hat{p}^{k}_{\htau}(\cdot,t^{n}_{\textnormal{f}})=\hat{p}^{k,n}_{h},\quad \hat{\vecu}^{k}_{\htau}(\cdot,t^{n}_{\textnormal{f}})=\hat{\vecu}^{k,n}_{h},\quad 0\leq n\leq N_{\textnormal{f}}.
 \end{equation}
An interesting result of the above reconstructions is given in the following Lemma (cf.~\cite{AHMED2018}):
\begin{lem}[Properties of $(\hat{p}^{k}_{\htau},\hat{\vecu}^{k}_{\htau})$]\label{Lemma:meanvalues}
 At each 
 iteration $k\geq 1$ of Algorithm~\ref{space_time_discrete_adaptive}, let 
 $(\widetilde{p}^{k}_{\htau},\widetilde{\vecu}^{k}_{\htau})$ be the 
 post-processed pressure and displacement,
 and  $(\hat{p}^{k}_{\htau},\hat{\vecu}^{k}_{\htau})$ be the 
 reconstructed pressure and displacement.
Then, for all $1\leq n\leq N_{\textn{f}}$, there holds
\label{mean_values_lemma}\begin{alignat}{2}\label{mean_values_lemma1}
\left(\partial_{t}\varphi(\hat{p}^{k,n}_{h},\hat{\vecu}^{k,n}_{h}),1\right)_{K}&=\left(\partial_{t}\varphi(\widetilde{p}^{k,n}_{h},\widetilde{\vecu}^{k,n}_{h}),1\right)_{K},&&\quad\forall K\in \Tau_{h}.
\end{alignat}
\end{lem}
\subsection{Equilibrated flux $\hat{\vecw}^{k}_{\htau}$ and 
stress $\hat{\sigmabo}^{k}_{\htau}$ reconstructions}
The second  ingredient for the derivation of our a posteriori error estimates is  to 
reconstruct  an equilibrated flux  $\hat{\vecw}^{k}_{\htau} \in P^{0}_{\tau}(\Hdiv)$,  locally conservative 
on the mesh $\Tau_h$:
\begin{equation}\label{eq:conservcondition_2_flow}
 \left(g^{n}-\partial_{t}\varphi(\hat{p}_{h}^{k,n},\hat{\vecu}_{h}^{k,n}) -\divecv\hat{\vecw}_{h}^{k,n},1\right)_{K}= 0,\quad\forall K\in\Tau_{h},
\end{equation}
and to  reconstruct  an equilibrated  stress  $\hat{\sigmabo}^{k}_{\htau}\in P^{0}_{\tau}(\Htensordiv)$,    locally conservative on the mesh $\Tau_h$:
   \begin{align}
   \left(\vf^{n}+\divecv\hat{\sigmabo}_{h}^{k,n},\mathbf{e}_{i}\right)_{K}=0,\quad i=1,\cdots,d,\quad\forall& K\in\Tau_{h}.\label{eq:conservcondition_2_stress}
 \end{align}
These reconstructions are based on solving 
local Neumann problems by mixed finite elements posed over patches of elements around mesh
vertices (cf.~\cite{ern2010posteriori,adaptnewernvohralik,RIEDLBECK20171593}). For each vertex $\veca\in \Vauh$, we introduce 
the   mixed  Raviart--Thomas finite element spaces posed on the patch domain $\omega_{\veca}$:
\bse\label{flowspaces_patches}\begin{alignat*}{2}
&Q^{\veca}_{h}:=\displaystyle\{q_{h}\in L^{2}(\omega_{\veca});\,\forall K\in \Tau^{\veca}_h,\,q_{h}|_{K}\in \amsmathbb{P}_{0}(K):(q_{h},1)_{\omega_{\veca}}=0 \},\\
&\mathbf{W}^{\veca}_{h}:=\displaystyle\{ \vecv_{h}\in \mathbf{H}(\divecvv,\omega_{\veca});\,\forall K\in \Tau^{\veca}_h,\,\vecv_{h}|_{K}\in \RTN_0(K):\vecv_{h}\cdot\vecn_{K}=0\,\textn{ on } \partial \omega_{\veca}\setminus\partial \Omega\}.
\end{alignat*}\ese
We then introduce the Arnold--Falk--Winther mixed finite elements spaces, posed on the patch 
domain $\omega_{\veca}$, for all $\veca\in \Vauh$:
\bse\label{mechspaces_patches}\begin{alignat*}{2}
&\mathbf{Q}^{\veca}_{h}:=\displaystyle\{\vecz_{h}\in \mathbf{L}^{2}(\omega_{\veca});\,\forall K\in \Tau^{\veca}_h,\,\vecz_{h}|_{K}\in \left[\amsmathbb{P}_{0}(K)\right]^{d}:(\vecz_{h},\mathbf{e}_{i})_{\omega_{\veca}}=0,\,i=1,\cdots,d \},\\
&\amsmathbb{W}^{\veca}_{h}:=\displaystyle\{ \vectau_{h}\in \amsmathbb{H}(\divecvv,\omega_{\veca});\,\forall K\in \Tau^{\veca}_h,\,\vectau_{h}|_{K}\in \left[\amsmathbb{P}_{1}(K)\right]^{d\times d}:\vectau_{h}\vecn_{K}=0\,\textn{ on } \partial \omega_{\veca}\setminus\partial \Omega\},\\
&\amsmathbb{Q}_{\textn{sk},h}^{\veca}:=\displaystyle\{\vecgamma_{h}\in [L^{2}(\omega_{\veca})]_{\textn{sk}}^{d\times d};\,\forall K\in \Tau^{\veca}_h,\,\vecgamma_{h}|_{K}\in [\amsmathbb{P}_{0}(K)]_{\textn{sk}}^{d\times d} \}.
\end{alignat*}\ese
We  obtain the  equilibrated velocity field $\hat{\vecw}^{k}_{\htau}$, by solving first for 
$(\hat{\vecw}_{\veca}^{k,n},q_{\veca}^{k,n})\in \mathbf{W}^{\veca}_{h}\times Q^{\veca}_{h}$, for all $1\leq n\leq N_{\textn{f}}$, such that 
 \bse\label{localflowreconst}\begin{alignat}{4}
&(\hat{\vecw}^{k,n}_{\veca}-\vecw^{k,n}_{h},\vecv)_{\omega_{\veca}}-(q^{k,n}_{\veca},\divecv\vecv)_{\omega_{\veca}}=0,&&\quad\forall \vecv\in \mathbf{W}^{\veca}_{h},\label{localflowreconst1}\\
&(\divecv\hat{\vecw}_{\veca}^{k,n},z)_{\omega_{\veca}}= (g^{n}-\partial_{t}\varphi(\hat{p}_{h}^{k,n},\hat{\vecu}_{h}^{k,n}),z)_{\omega_{\veca}},&&\quad\forall z\in Q^{\veca}_{h}.\label{localflowreconst2}
\end{alignat}
Then set  
\begin{equation}
           \hat{\vecw}_{h}^{k,n}=\displaystyle\sum_{\veca\in \Vauh}\hat{\vecw}_{\veca}^{k,n}.
\end{equation}\ese
 For the   equilibrated stress $\hat{\sigmabo}_{\htau}^{k}$, we   solve  local Neumann
mechanics problems by mixed finite elements, with weakly symmetric stress: 
find $(\hat{\sigmabo}_{\veca}^{k,n},
\vecz_{\veca}^{k,n},\betabo_{\veca}^{k,n})\in  \amsmathbb{W}^{\veca}_{h}\times \mathbf{Q}^{\veca}_{h}\times \amsmathbb{Q}_{\textn{sk},h}^{\veca}$,
for all $1\leq n\leq N_{\textn{f}}$, such that

\bse\label{localstressreconst}\begin{alignat}{4}
&(\hat{\sigmabo}^{k,n}_{\veca}-\sigmabo^{k,n}_{h},\vectau)_{\omega_{\veca}}+
(\vecz^{k,n}_{\veca},\divecv\vectau)_{\omega_{\veca}}+(\betabo_{\veca}^{k,n},\vectau)_{\omega_{\veca}}=0,&&\quad\forall \vectau\in \amsmathbb{W}^{\veca}_{h},\label{localstressreconst1}\\
&(\divecv\hat{\sigmabo}_{\veca}^{k,n},\vecv)_{\omega_{\veca}}= (-\vf^{n},\vecv)_{\omega_{\veca}},
&&\quad\forall \vecv\in \mathbf{Q}^{\veca}_{h},
\label{localstressreconst2}\\
&(\hat{\vecsigma}_{\veca}^{k,n},\vecgamma)_{\omega_{\veca}}=0,&&\quad\forall \vecgamma\in \amsmathbb{Q}_{\textn{sk},h}^{\veca}.\label{localstressreconst3}
\end{alignat}
Then set  
\begin{equation}
\hat{\vecsigma}_{h}^{k,n}=\displaystyle\sum_{\veca\in \Vauh}\hat{\vecsigma}_{\veca}^{k,n}.
\end{equation}\ese
    The above local problems are well-posed owing to the properties of 
    mixed  finite   elements~(cf. \cite{arnold1985mixed,arnold2007mixed}).  We can easily observe that the computational cost of  the flux and stress reconstructions can be substantially reduced by pre-processing, a  step that is fully parallelizable.

\section{The a posteriori error estimates}\label{sec:Aposteriori}
We derive in this section, based on the previous constructions, 
a posteriori error estimates for the solution of Algorithm~\ref{space_time_discrete} or \ref{multi-rate_discrete}. 
This is done by bounding an energy  error between the exact weak solution $(p,\vecu)$ of 
problem~\eqref{Main_problem_model} and 
the approximate  solution 
$(\widetilde{p}^{k}_{\htau},\widetilde{\vecu}^{k}_{\htau})$, by a guaranteed and 
fully computable upper bound, and this at each iteration $k\geq 1$ of Algorithm~\ref{space_time_discrete} or \ref{multi-rate_discrete}.
\subsection{The error measure}
The first question in a posteriori error estimates is that of the error measure; here we will in particular rely on
~\cite[Theorem 6.2]{AHMED2018},  where an energy-type error in the pressure and displacement is shown to be bounded by the   dual norm of the residuals, and where the Biot's consolidation  equations~\eqref{Weak_problem_mixed_model}, \eqref{Main_problem_model_IC} are discretized
 using MFE method in-space and with a backward Euler scheme in-time and solved monolithically. To this aim, and   for all times $t\in(0,T]$, we let
\begin{equation*}
  Q_{t} \eq L^{2}(0,t;L^{2}(\Omega)),\quad X_{t} \eq L^{2}(0,t;H_{0}^{1}(\Omega)),
  \quad X^{\prime}_{t} \eq L^{2}(0,t;H^{-1}(\Omega)),\quad\mathbf{Z}_{t}\eq H^{1}(0,t;\mathbf{H}^{1}_{0}(\Omega)),
\end{equation*}
and introduce the energy space
$$\mathcal{E}_{t}\eq\left\{(p,\vecu)\,|\, p \in X_{t},\quad \vecu\in \mathbf{Z}_{t},\textn{ such that }\partial_{t}\varphi(p,\vecu)\in 
X^{\prime}_{t}\right\}.$$
Then, we introduce  a weak formulation 
of~\eqref{Main_problem_model}: find $(p,\vecu)\in  \mathcal{E}_{T}$ such that
$p(\cdot,0)=p_0$ and $\vecu(\cdot,0)=\vecu_0$ and such that
\bse \label{weak_primal_form}\begin{alignat}{3}\label{weak_primal_form1}
 &\int_{0}^{T}\langle\partial_{t}\varphi(p,\vecu),q\rangle\dt+\int_{0}^{T}(\vK \nabla p,\nabla q)\dt=\int_{0}^{T}(g,q)\dt,&&\quad\forall q\in X_{T},\\
\label{weak_primal_form3} &\int_{0}^{T}(\vectheta(\vecu),\epsb(\vecv))\dt-\alpha\int_{0}^{T}(p,\divecv\vecv)\dt=-\int_{0}^{T}(\vf,\vecv)\dt,&&\quad\forall \vecv\in \mathbf{X}_{T},
 \end{alignat}\ese
where  $\langle\cdot,\cdot\rangle$ denotes the duality pairing between $H^{-1}(\Omega)$ and $H^{1}_{0}(\Omega)$. 
The existence and uniqueness of the solution to this problem was addressed in~\cite{AHMED2018}.
Still following~\cite{AHMED2018}, we introduce the following energy-type error measure
\bse\begin{alignat}{2}
\nonumber  \|(p-\hat{p}^{k}_{\htau},\vecu-\hat{\vecu}^{k}_{\htau})\|^{2}_{\textn{en}} &\eq 
\dfrac{1}{2}\|(p-\hat{p}^{k}_{\htau},\vecu-\hat{\vecu}^{k}_{\htau})\|^{2}_{\natural_{T}}
   +\dfrac{1}{2}\|\varphi(p-\hat{p}^{k}_{\htau},\vecu-\hat{\vecu}^{k}_{\htau})\|^{2}_{X^{\prime}_{T}}\qquad\qquad\qquad\qquad\;\\
   \nonumber&\qquad+2c_{0}\int_{0}^{T}\Bigg(||p-\hat{p}^{k}_{\htau}||_{Q_{t}}^{2}
   +\int_{0}^{t}||p-\hat{p}^{k}_{\htau}||_{Q_{s}}^{2}e^{t-s} \textn{d}s\Bigg)\textn{d}t\\
  \label{eq:newenergynorm}
    &\qquad\quad   +\int_{0}^{T}\Bigg(||\vecu-\hat{\vecu}^{k}_{\htau}||_{\Xi_{t}}^{2}
   +\int_{0}^{t}||\vecu-\hat{\vecu}^{k}_{\htau}||_{\Xi_{s}}^{2}e^{t-s} \textn{d}s\Bigg)\textn{d}t,
\end{alignat}
where
\begin{alignat}{2}
\label{energy_like_norm}&\|(p-\hat{p}_{\htau}^{k},\vecu-\hat{\vecu}^{k}_{\htau})\|^{2}_{\natural_{t}} \eq  
c_{0}||p-\hat{p}^{k}_{\htau}||_{Q_{t}}^{2}+\dfrac{1}{2}||\vecu-\hat{\vecu}^{k}_{\htau}||_{\Xi_{t}}^{2}+\dfrac{1}{2}||\varphi(p-\hat{p}^{k}_{\htau},\vecu-\hat{\vecu}^{k}_{\htau})(t)||_{H^{-1}(\Omega)}^{2},\\
& ||\vecu-\hat{\vecu}^{k}_{\htau}||^{2}_{\Xi_{t}}\eq2\mu||\epsb(\vecu-\hat{\vecu}^{k}_{\htau})||^{2}_{Q_{t}}+\lambda||\divecv(\vecu-\hat{\vecu}^{k}_{\htau})||^{2}_{Q_{t}}.
\end{alignat}\ese
The above norms are  well-defined owing to the properties of the weak solution 
$(p,\vecu)$ and the reconstructed functions
$(\hat{p}^{k}_{\htau},\hat{\vecu}^{k}_{\htau})$, i.e.,  we have both $(p,\vecu)$ 
 and $(\hat{p}^{k}_{\htau},\hat{\vecu}^{k}_{\htau})$  in  $\mathcal{E}_{T}$.
\subsection{The error estimators} 
 Before formulating our estimators,  we define the
broken Sobolev space $H^{1}(\Tau_{h})$ as the space of all functions $v\in
L^{2}(\Omega)$ such that $v|_{K}\in H^{1}(K)$, for all $K\in\Tau_{h}$. The energy semi-norm on $H^{1}(\Tau_{h})$
is given by
\begin{equation}\label{H1discretenorm}
 |||v |||^{2}\eq \sum_{K\in\Tau_{h}}|||v |||^{2}_{K}=\sum_{K\in\Tau_{h}}||\vK^{\frac{1}{2}}\nabla v||^{2}_{K}, \qquad\forall v\in H^{1}(\Tau_{h}),
\end{equation}
where the sign $\nabla$  denote the element-wise gradient, i.e., the
gradient of a function restricted to a mesh element $K\in\Tau_{h}$.
The energy norm in $\mathbf{L}^{2}(\Omega)$ is given by 
\begin{equation}\label{L2discretenorm}
 ||\vecv ||^{2}_{\star}\eq \sum_{K\in\Tau_{h}}||\vecv ||^{2}_{\star,K}=\sum_{K\in\Tau_{h}}||\vK^{-\frac{1}{2}} \vecv||^{2}_{K}, \qquad\forall \vecv\in \mathbf{L}^{2}(\Tau_{h}).
\end{equation}
We also recall the Poincar\'{e} inequality: 
\begin{equation}\label{Poincare_ineq}
 \|q-q_{K}\| \leq C_{P,K} h_{K}\|\nabla q \|_{K},\quad \forall q\in H^{1}(K),
\end{equation}
 where $q_K$ is the mean value of the function $q$ on the element $K$ given by $q_K=\int_{K}q\dx/|K|$ and
 $C_{P,K}=1/\pi$ whenever the element $K$ is convex. In what follows,  we denote respectively by $c_{\vK,K}$ and $C_{\vK,K}$  the smallest and the largest eigenvalue of
the tensor $\vK$ in $K \in \Tau_{h}$. 
We  introduce the \textit{local residual estimators}
\bse\begin{align}
  \label{Residual_estimatorr_P}
  \eta^{k,n}_{\textn{R,P},K}
    &\eq C_{P,K}c_{\vK,K}^{-\frac{1}{2}} h_{K} \|g^{n}-\partial_{t}\varphi(\hat{p}^{k,n}_{h},\hat{\vecu}^{k,n}_{h})-\divecv\hat{\vecw}_{h}^{k,n}\|_{K}, \quad K\in\Tau_{h}, \\
  \label{Residual_estimatorr_U}
  \eta^{k,n}_{\textn{R,U},K}
   & \eq C_{P,K} h_{K}\|\divecv\hat{\sigmabo}^{k,n}_{h}+\vf^{n}\|_{K}, \quad K\in\Tau_ {h},
\end{align}\ese
 the \textit {flux estimators}
\bse\label{Disc_estimator_eq}\begin{align}
  \label{Disc_estimator_eq_P}
  \eta^{k,n}_{\textn{F,P},K}(t)
   & \eq ||\hat{\vecw}^{k,n}_{h}+\vK\nabla{\hat{p}^{k}}_{\htau}(t)||_{\star,K}, \quad K\in\Tau_{h},\,\, t\in I_{\textn{f}}^{n},\\
    \label{Disc_estimator_eq_U}
  \eta^{k,n}_{\textn{F,U},K}(t)
   & \eq ||\hat{\sigmabo}^{k,n}_{h}-\sigmabo({\hat{p}^{k}}_{\htau},{\hat{\vecu}^{k}}_{\htau})(t)||_{K}, \quad K\in\Tau_{h},\,\, t\in I_{\textn{f}}^{n},
\end{align}\ese
the  \textit{pressure nonconformity estimators}
\bse    \begin{align}
    \label{Disc_estimator_eq_NC1P}
  &\eta^{n}_{\textn{NC1,P},K}(t)
    \eq \left(\dfrac{c_{0}}{2}\right)^{\frac{1}{2}}\|(\widetilde{p}^{k}_{h}-\hat{p}^{k}_{h})(t)\|_{K}, \quad K\in\Tau_ {h},\quad t\in I_{\textn{f}}^{n},\\
      &\eta^{k,n}_{\textn{NC2,P},K}
     \eq     c_{0}\sqrt{2}\dfrac{h_{K}c_{\vK,K}^{-\frac{1}{2}}}{3\pi} \left\{\|\widetilde{p}_{h}^{k,n}-\hat{p}^{k,n}_{h}\|^{2}_{K}+\|\widetilde{p}_{h}^{k,n-1}-\hat{p}^{k,n-1}_{h}\|^{2}_{K}\right\}^{\frac{1}{2}}, \quad K\in\Tau_ {h},\\
  & \eta^{k}_{\textn{NCF,P},K}
     \eq c_{0}\dfrac{h_{K}c_{\vK,K}^{-\frac{1}{2}}}{2\pi}\|(\widetilde{p}^{k}_{h}-\hat{p}^{k}_{h})(\cdot,T)\|_{K}, \quad K\in\Tau_ {h},
\end{align}
and  the  \textit{displacement nonconformity estimators}
\begin{align}
 \label{Disc_estimator_eq_NC1U}
  &\eta^{n}_{\textn{NC1,U},K}(t)
    \eq \dfrac{1}{2}\left\{2\mu\|\epsb(\widetilde{\vecu}^{k}_{h}-\hat{\vecu}^{k}_{h})(t)\|_{K}^2+\lambda\|\divecv(\widetilde{\vecu}^{k}_{h}-\hat{\vecu}^{k}_{h})(t)\|_{K}^2\right\}^{\frac{1}{2}}, \quad K\in\Tau_ {h},\quad t\in I_{\textn{f}}^{n},\\
 & \eta^{k,n}_{\textn{NC2,U},K}
     \eq  \alpha\sqrt{2}\dfrac{h_{K}c_{\vK,K}^{-\frac{1}{2}}}{3\pi}\left\{\|\divecv(\widetilde{\vecu}^{k,n}_{h}-\hat{\vecu}^{k,n}_{h})\|^{2}_{K}
     +\|\divecv(\widetilde{\vecu}^{k,n-1}_{h}-\hat{\vecu}^{k,n-1}_{h})\|^{2}_{K}\right\}^{\frac{1}{2}}, \quad K\in\Tau_ {h},\\      
  &   \eta^{k}_{\textn{NCF,U},K}
     \eq\alpha\dfrac{h_{K}c_{\vK,K}^{-\frac{1}{2}}}{2\pi}\|\divecv(\widetilde{\vecu}^{k}_{h}-\hat{\vecu}^{k}_{h})(\cdot,T)\|_{K}, \quad K\in\Tau_ {h}.
\end{align}\ese
Therefrom, we introduce the \textit{global versions} by
\bse\label{globa_space_tim_estimators}
\begin{align}
&  \eta^{k,n}_{\textn{J}} \eq \left\{\int_{I_{\textn{f}}^{n}}\sum_{K\in\Tau_{h}} \left(\eta^{k,n}_{\textn{R,J},K}
  + \eta^{k,n}_{\textn{F,J},K}(t) \right)^{2}\dt\right\}^{\frac{1}{2}},
     \quad 1\leq n\leq N_{\textnormal{f}},\quad \textn{J}=\textn{P},\,\textn{U}, \label{local_estim_sum_R_Dis_i}\\
&\eta^{k,n}_{\textn{NC1,J}} \eq \left\{\int_{I_{\textn{f}}^{n}}\sum_{K\in\Tau_{h}} \left(\eta^{k,n}_{\textn{NC1,J},K}(t) \right)^{2}\dt\right\}^{\frac{1}{2}},
     \quad 1\leq n\leq N_{\textnormal{f}},\quad \textn{J}=\textn{P},\,\textn{U},\label{local_estim_sum_R_Dis_iii}\\
     \label{local_estim_sum_R_Dis_iiiii}
  &\eta^{k,n}_{\textn{NC2,J}} \eq \left\{\int_{I_{\textn{f}}^{n}}\sum_{K\in\Tau_{h}} \left(\eta^{k,n}_{\textn{NC2,J},K} \right)^{2}\dt\right\}^{\frac{1}{2}}, \quad 1\leq n\leq N_{\textnormal{f}},\quad \textn{J}=\textn{P},\,\textn{U},\\
    &\eta^{k}_{\textn{NCF,\textn{J}}} \eq \left\{\sum_{K\in\Tau_{h}} \left(\eta^{k}_{\textn{NCF,\textn{J}},K} \right)^{2}\right\}^{\frac{1}{2}},\quad \textn{J}=\textn{P},\,\textn{U}.
    \end{align}\ese
\subsection{Guaranteed reliability}
We now provide a guaranteed estimate on the total error in particular valid on each iteration of Algorithm~\ref{space_time_discrete} or \ref{multi-rate_discrete}.  This result extend the results of our previous work \cite[Section~5]{AHMED2018}, where a computable guaranteed
 bound on the energy-type error between the exact solution and its approximation with an exact
solver has been derived. 
\begin{thm}[Global-in-time a posteriori error estimate]\label{thm:main_theorem}
 Let $(p,\vecu)$ be the weak solution of~\eqref{weak_primal_form}. At an  arbitrary  
 iteration $k\geq 1$ of Algorithm~\ref{space_time_discrete} or \ref{multi-rate_discrete}, let 
 $(\widetilde{p}^{k}_{\htau},\widetilde{\vecu}^{k}_{\htau})$ be the post-processed pressure and displacement of 
 subsection~\ref{sec:post-processing},  $(\hat{p}^{k}_{\htau},\hat{\vecu}^{k}_{\htau})$ be the 
 reconstructed pressure and displacement   and 
 $(\hat{\vecw}^{k}_{\htau},\hat{\sigmabo}^{k}_{\htau})$ be the reconstructed 
 flux and stress of   section~\ref{sec:reconstructions}. Then, there holds
 \begin{equation}
  \label{eq:est}
  \|(p-\widetilde{p}^{k}_{\htau},\vecu-\widetilde{\vecu}^{k}_{\htau})\|_{\textn{en}} \leq \eta^{k}_{\textn{P}}+ \eta^{k}_{\textn{U}}   +  \eta^{k}_{\textn{NC,P}}+ 
  \eta^{k}_{\textn{NC,U}},
\end{equation}
where
\bse\label{global_versions_estimators}\begin{alignat}{2} \nonumber
 &   \eta^{k}_{\textn{J}} \eq {}  \sqrt{\dfrac{L_{\textn{J}}}{2}}\Bigg\{ \sum_{n=1}^{N_{\textnormal{f}}}\left(\eta^{k,n}_{\textn{J}}\right)^2
  +2\sum_{n=1}^{N_{\textnormal{f}}}\tau^{n}\sum_{l=1}^{n}\big(\eta^{k,l}_{\textn{J}}\big)^{2}\\&\qquad\qquad\qquad\quad
 +2\sum_{n=1}^{N_{\textnormal{f}}}\sum_{l=1}^{n}J_{nl}\Big(\sum_{q=1}^{l}\left(\eta^{k,q}_{\textn{J}}\right)^{2}\Big)\Bigg\}^{\frac 1 2},\quad \textn{J}=\textn{P},\,\textn{U},\\
   &\nonumber
    \eta^{k}_{\textn{NC,J}} \eq {}\Bigg\{\sum_{n=1}^{N_{\textnormal{f}}} \{\left(\eta^{k,n}_{\textn{NC1,J}}\right)^2+\left(\eta^{k,n}_{\textn{NC2,J}}\right)^2\} +4\sum_{n=1}^{N_{\textnormal{f}}}\tau^{n}\sum_{l=1}^{n}\big(\eta^{k,l}_{\textn{NC1,J}}\big)^{2} \\
 \label{eq:eta:totP} &\qquad\qquad\qquad\qquad\quad
 +4\sum_{n=1}^{N_{\textnormal{f}}}\sum_{l=1}^{n}J_{nl}\Big(\sum_{q=1}^{l}\left(\eta^{k,q}_{\textn{NC1,J}}\right)^{2}\Big)+\left(\eta^{k}_{\textn{NCF,J}}\right)^{2}\Bigg\}^{\frac 1 2},\quad \textn{J}=\textn{P,\,U}.
\end{alignat}\ese
Notice that we have set $L_{\textn{P}}=1$ and $L_{\textn{U}}=\frac{1}{\mu}$, and for $1\leq n,l\leq N_{\textn{f}}$,
$$
J_{nl} \eq \int_{I_{\textn{f}}^{n}}\int_{I_{\textn{f}}^{l}} e^{t-s}\textn{d}s\textn{d}t.
$$
\end{thm}
\begin{proof}
Recalling~\eqref{eq:newenergynorm}, and~\eqref{initial_data_satisfied}, we have from~\cite[Theorem 6.2]{AHMED2018}, 
for any given couple 
$(\hat{p}_{\htau},\hat{\vecw}_{\htau})\in \mathcal{E}_{T}$,
\begin{alignat}{2}
\nonumber&\|(p-\hat{p}_{\htau},\vecu-\hat{\vecu}_{\htau})\|^{2}_{\textn{en}}&&\leq
\dfrac{1}{2}||\mathcal{R}_{\textn{P}}(\hat{p}_{\htau},\hat{\vecu}_{\htau})||^{2}_{X^{\prime}_{T}}+
 \dfrac{1}{2\mu}||\mathcal{R}_{\textn{U}}(\hat{p}_{\htau},\hat{\vecu}_{\htau})||_{\mathbf{X}_{T}^{\prime}}^{2}\\
\nonumber&& & +\int_{0}^{T}\Bigg(||\mathcal{R}_{\textn{P}}(\hat{p}_{\htau},\hat{\vecu}_{\htau})||^{2}_{X^{\prime}_{t}}+\int_{0}^{t}||\mathcal{R}_{\textn{P}}(\hat{p}_{\htau},\hat{\vecu}_{\htau})||^{2}_{X^{\prime}_{s}}e^{t-s} \textn{d}s\Bigg)\textn{d}t\\
\label{first_estimate}&& & +\dfrac{1}{\mu}\int_{0}^{T}\Bigg(||\mathcal{R}_{\textn{U}}(\hat{p}_{\htau},\hat{\vecu}_{\htau})||^{2}_{\mathbf{X}^{\prime}_{t}}+\int_{0}^{t}||\mathcal{R}_{\textn{U}}(\hat{p}_{\htau},\hat{\vecu}_{\htau})||^{2}_{\mathbf{X}^{\prime}_{s}}e^{t-s} \textn{d}s\Bigg)\textn{d}t; 
\end{alignat}
  featuring the residuals  $\mathcal{R}_{\textn{P}}(\hat{p}_{\htau},\hat{\vecu}_{\htau})\in X^{\prime}_{T}$ and  $\mathcal{R}_{\textn{U}}(\hat{p}_{\htau},\hat{\vecu}_{\htau})\in \mathbf{X}^{\prime}_{T}$ of the weak formulation~\eqref{weak_primal_form}: for all 
$q\in X_{T}$ and $\vecv\in \mathbf{X}_{T}$,
\begin{alignat}{3}\label{residual_P}
 &\langle \mathcal{R}_{\textn{P}}(\hat{p}_{\htau},\hat{\vecu}_{\htau}),q\rangle_{X^{\prime}_{T},X_{T}}:=\int_{0}^{T}(g,q)\dt-\int_{0}^{T}\langle\partial_{t}\varphi(\hat{p}_{\htau},\hat{\vecu}_{\htau}),q\rangle\dt-\int_{0}^{T}(\vK \nabla \hat{p}_{\htau},\nabla q)\dt,\\
 \label{residual_U}&\langle \mathcal{R}_{\textn{U}}(\hat{p}_{\htau},\hat{\vecu}_{\htau}),\vecv\rangle_{\mathbf{X}^{\prime}_{T},\mathbf{X}_{T}}:=\int_{0}^{T}(\vf,\vecv)\dt+\int_{0}^{T}(\vectheta(\hat{\vecu}_{\htau}),\epsb(\vecv))\dt-\alpha\int_{0}^{T}(\hat{p}_{\htau},\divecv\vecv)\dt.
 \end{alignat}
  The dual norms of the residuals are 
  given by
 \begin{alignat}{3}\label{residual_PNORM}
 &||\mathcal{R}_{\textn{P}}(\hat{p}_{\htau},\hat{\vecu}_{\htau})||_{X^{\prime}_{T}}:=\sup_{\underset{||q||_{X_{T}}=1}{ q\in X_{T}}}\langle \mathcal{R}_{\textn{P}}(\hat{p}_{\htau},\hat{\vecu}_{\htau}),q\rangle_{X^{\prime}_{T},X_{T}},\\
 \label{residual_UNORM}&||\mathcal{R}_{\textn{U}}(\hat{p}_{\htau},\hat{\vecu}_{\htau})||_{\mathbf{X}^{\prime}_{T}}:=\sup_{\underset{||\vecv||_{\mathbf{X}_{T}}=1}{ \vecv\in \mathbf{X}_{T}}}\langle \mathcal{R}_{\textn{U}}(\hat{p}_{\htau},\hat{\vecu}_{\htau}),\vecv\rangle_{\mathbf{X}^{\prime}_{T},\mathbf{X}_{T}}.
 \end{alignat} 
 At each iteration $k\geq 1$, the approximate solution  $(\widetilde{p}^{k}_{\htau},\widetilde{\vecu}^{k}_{\htau})$ is not an element of 
 $\mathcal{E}_{T}$, contrarily to   the reconstructed solution, i.e.,    
 $(\hat{p}^{k}_{\htau},\hat{\vecu}^{k}_{\htau})\in \mathcal{E}_{T}$. Thus, to use~\eqref{first_estimate}, we apply   the triangle inequality to get  
   \begin{equation}\label{estimate_triang_ineq}
  ||(p-\widetilde{p}^{k}_{\htau},\vecu-\widetilde{\vecu}^{k}_{\htau})||_{\textn{en}} \leq ||(p-\hat{p}_{\htau}^{k},\vecu-\hat{\vecu}^{k}_{\htau})||_{\textn{en}}  + 
  ||(\hat{p}^{k}_{\htau}-\widetilde{p}^{k}_{\htau},\hat{\vecu}^{k}_{\htau}-\widetilde{\vecu}^{k}_{\htau})||_{\textn{en}},
\end{equation}
where  we can bound the first term of the right-hand side using~\eqref{first_estimate}. What remains is to give  a computable upper bound for the residuals $||\mathcal{R}_{\textn{P}}(\hat{p}^{k}_{\htau},\hat{\vecu}^{k}_{\htau})||_{X^{\prime}_{T}}$ and $||\mathcal{R}_{\textn{U}}(\hat{p}^{k}_{\htau},\hat{\vecu}^{k}_{\htau})||_{\mathbf{X}^{\prime}_{T}}$ together with $||(\hat{p}^{k}_{\htau}-\widetilde{p}^{k}_{\htau},\hat{\vecu}^{k}_{\htau}-\widetilde{\vecu}^{k}_{\htau})||_{\textn{en}}$ and then combine these results.\\

\textbf{1) A computable upper bound for $||\mathcal{R}_{\textn{P}}(\hat{p}^{k}_{\htau},\hat{\vecu}^{k}_{\htau})||_{X^{\prime}_{T}}$ and $||\mathcal{R}_{\textn{U}}(\hat{p}^{k}_{\htau},\hat{\vecu}^{k}_{\htau})||_{\mathbf{X}^{\prime}_{T}}$.} Proceeding as 
in~\cite{ahmed:hal-01687026,RIEDLBECK20171593}, adding  $(\hat{\vecw}^{k,n}_{h},\nabla q)$ 
to~\eqref{residual_P}, choosing $q\in X_{T}$ with $||q||_{X_{T}}=1$
and  applying    the  Green theorem, and using~\eqref{eq:conservcondition_2_flow},  we obtain
\begin{alignat}{2}\label{residual_PP}
 \nonumber&\langle \mathcal{R}_{\textn{P}}(\hat{p}^{k}_{\htau},\hat{\vecu}^{k}_{\htau}),q\rangle_{X^{\prime}_{T},X_{T}}\\ 
 \nonumber&=\sum_{n=1}^{N_{\textn{f}}}\int_{I_{\textn{f}}^{n}}\{\left(g^{n}-\partial_{t}\varphi(\hat{p}^{k,n}_{h},\hat{\vecu}^{k,n}_{h})-\divecv\hat{\vecw}^{k,n}_{h},q\right)-(\hat{\vecw}^{k,n}_{h}+\vK \nabla \hat{p}^{k}_{\htau}),\nabla q)\}\dt,\\
&  =\sum_{n=1}^{N_{\textn{f}}}\int_{I^{n}_{\textn{f}}}\{\left(g^{n}-\partial_{t}\varphi(\hat{p}^{k,n}_{h},\hat{\vecu}^{k,n}_{h})-\divecv\hat{\vecw}^{k,n}_{h},q-q_{K}\right)-(\hat{\vecw}^{k,n}_{h}+\vK \nabla \hat{p}^{k}_{\htau}),\nabla q)\}\dt.
 \end{alignat}
 Then, it can be inferred using \eqref{residual_PNORM} and the Poincar\'{e} inequality~\eqref{Poincare_ineq}
 followed by  Cauchy-Schwarz inequality that
 \begin{alignat}{1}
 \label{boundRP}||\mathcal{R}_{\textn{P}}(\hat{p}^{k}_{\htau},\hat{\vecu}^{k}_{\htau})||_{X^{\prime}_{T}}
 \leq \displaystyle\left\{\sum_{n=1}^{N_{\textn{f}}}\int_{I^{n}_{\textn{f}}}\sum_{K\in\Tau_{h}}\left(\eta^{k,n}_{\textn{R,P},K}+\eta^{k,n}_{\textn{F,P},K}(t)\right)^{2}\dt\right\}^{\frac{1}{2}}.
\end{alignat}
We repeat the same steps for~\eqref{residual_U}, by adding and subtracting $(\hat{\vecsigma}^{k,n}_{\htau},\nabla\vecv)$
(we replace $(\hat{\vecsigma}^{k,n}_{\htau},\epsb(\vecv))  $ by  
$(\hat{\vecsigma}^{k,n}_{\htau},\nabla\vecv)$ due to symmetry), using~\eqref{eq:conservcondition_2_stress}
and applying the Poincar\'{e} inequality~\eqref{Poincare_ineq} together  
with the Cauchy-Schwarz inequality and definition~\eqref{residual_UNORM}, we obtain 
\begin{alignat}{2}
\label{boundRU}||\mathcal{R}_{\textn{U}}(\hat{p}^{k}_{\htau},\hat{\vecu}^{k}_{\htau})||_{\mathbf{X}^{\prime}_{T}}&&\leq \displaystyle\left\{\sum_{n=1}^{N_{\textn{f}}}\int_{I_{\textn{f}}^{n}}\sum_{K\in\Tau_{h}}\left(\eta^{k,n}_{\textn{R,U},K}+\eta^{k,n}_{\textn{F,U},K}(t)\right)^{2}\dt\right\}^{\frac{1}{2}}.
\end{alignat}
Replacing~\eqref{boundRP} and~\eqref{boundRU} in~\eqref{first_estimate},  so we are left to bound the third and fourth terms of the 
right-hand side of~\ref{first_estimate}. Using the fact that 
$||\mathcal{R}_{\textn{U}}(\hat{p}^{k}_{\htau},\hat{\vecu}^{k}_{\htau})||_{\mathbf{X}_{t}^{\prime}}^{2}$ (also $||\mathcal{R}_{\textn{P}}(\hat{p}^{k}_{\htau},\hat{\vecu}^{k}_{\htau})||^{2}_{X^{\prime}_{t}}$)
is a nondecreasing function of the time $t$, we easily obtain from \eqref{boundRP}-\eqref{boundRU},
\bse\begin{alignat}{3}
& \int_{0}^{T}||\mathcal{R}_{\textn{P}}(\hat{p}^{k}_{\htau},\hat{\vecu}^{k}_{\htau})||^{2}_{X^{\prime}_{t}}\textn{d}t
 \leq \sum_{n=1}^{N_{\textn{f}}}\int_{I_{\textn{f}}^{n}}||\mathcal{R}_{\textn{P}}(\hat{p}^{k}_{\htau},\hat{\vecu}^{k}_{\htau})||^{2}_{X^{\prime}_{t^{n}_{\textn{f}}}}\leq 
 \sum_{n=1}^{N_{\textn{f}}}\tau^{n}_{\textn{f}}\left(\eta^{k,n}_{\textn{P}}\right)^{2},\\
&\int_{0}^{T}||\mathcal{R}_{\textn{U}}(\hat{p}^{k}_{\htau},\hat{\vecu}^{k}_{\htau})||^{2}_{\mathbf{X}^{\prime}_{t}}\textn{d}t
 \leq \sum_{n=1}^{N_{\textn{f}}}\int_{I_{\textn{f}}^{n}}||\mathcal{R}_{\textn{U}}(\hat{p}^{k}_{\htau},\hat{\vecu}^{k}_{\htau})||^{2}_{\mathbf{X}^{\prime}_{t_{\textn{f}}^{n}}}\leq 
 \sum_{n=1}^{N_{\textn{f}}}\tau^{n}_{\textn{f}}\left(\eta^{k,n}_{\textn{U}}\right)^{2}.
 \end{alignat}\ese
 In a similar way, we infer
 \begin{alignat}{3}
 \nonumber &\int_{0}^{T}\int_{0}^{t}||\mathcal{R}_{\textn{P}}(\hat{p}^{k}_{\htau},\hat{\vecu}^{k}_{\htau})||^{2}_{X^{\prime}_{s}}e^{t-s}
  \textn{d}s\textn{d}t\leq \sum_{n=1}^{N_{\textn{f}}}\int_{I_{\textn{f}}^{n}} 
  \sum_{l=1}^{n}\int_{I_{\textn{f}}^{l}}
  ||\mathcal{R}_{\textn{P}}(\hat{p}^{k}_{\htau},\hat{\vecu}^{k}_{\htau})||^{2}_{X^{\prime}_{t_{\textn{f}}^{l}}}e^{t-s}\textn{d}s\textn{d}t,\\
 \nonumber&\qquad\qquad\leq \sum_{n=1}^{N_{\textn{f}}}\int_{I_{\textn{f}}^{n}} 
    \sum_{l=1}^{n}\left\{\int_{I_{\textn{f}}^{l}} \sum_{l=1}^{n}\left(\eta^{k,n}_{\textn{P}}\right)^{2}e^{t-s}\textn{d}s\right\},\\
 \label{boundRP_exp}&\qquad\qquad= \sum_{n=1}^{N_{\textn{f}}} \sum_{l=1}^{n}
 \left\{\int_{I_{\textn{f}}^{n}}\int_{I_{\textn{f}}^{l}}e^{t-s}\textn{d}s\textn{d}t\right\}\times
  \left\{\sum_{l=1}^{n}\left(\eta^{k,n}_{\textn{P}}\right)^{2}\right\}=
  \sum_{n=1}^{N_{\textn{f}}} \sum_{l=1}^{n}
 J_{nl}
  \left\{\sum_{q=1}^{l}\left(\eta^{k,q}_{\textn{P}}\right)^{2}\right\},
 \end{alignat}
 and similarly
 \begin{alignat}{3}
\label{boundRU_exp} &\int_{0}^{T}\int_{0}^{t}||\mathcal{R}_{\textn{U}}(\hat{p}^{k}_{\htau},\hat{\vecu}^{k}_{\htau})||^{2}_{\mathbf{X}^{\prime}_{s}}e^{t-s}
  \textn{d}s\textn{d}t\leq
  \sum_{n=1}^{N_{\textn{f}}} \sum_{l=1}^{n}
 J_{nl}
  \left\{\sum_{q=1}^{n}\left(\eta^{k,q}_{\textn{U}}\right)^{2}\right\}.
   \end{alignat}
 We use \eqref{boundRP}--\eqref{boundRU_exp}  in~\eqref{first_estimate}, thus we bound 
the  first term of~\eqref{estimate_triang_ineq}:
\begin{equation}\label{estimate_triang_ineq21}
||(p-\hat{p}_{\htau}^{k},\vecu-\hat{\vecu}^{k}_{\htau})||_{\textn{en}}\leq \eta^{k}_{\textn{P}}+\eta^{k}_{\textn{U}}.
\end{equation}

\textbf{2) A computable upper bound to $||(\hat{p}^{k}_{\htau}-\widetilde{p}^{k}_{\htau},\hat{\vecu}^{k}_{\htau}-\widetilde{\vecu}^{k}_{\htau})||_{\textn{en}}$.} To bound this term presenting the nonconformity estimator, we proceed  as in ~\cite[Theorem~5.3 $\&$ Lemma~5.7]{AHMED2018}, we promptly arrive to
\begin{alignat}{2}
\label{estimate_triang_ineq22}  ||(\hat{p}^{k}_{\htau}-\widetilde{p}^{k}_{\htau},\hat{\vecu}^{k}_{\htau}-\widetilde{\vecu}^{k}_{\htau})||_{\textn{en}}\leq \eta^{k}_{\textn{NC,P}}+\eta^{k}_{\textn{NC,U}}.
\end{alignat}
The estimate~\eqref{eq:est} is obtained  by replacing~\eqref{estimate_triang_ineq21} and~\eqref{estimate_triang_ineq22} 
in~\eqref{estimate_triang_ineq}.
\end{proof}
\subsection{An a posteriori error estimate distinguishing the space,   time  and fixed-stress coupling
errors}\label{sec:ErrorComponents}
Our  goal in this section is to distinguish the different error components. 
Particularly, we  separate the iterative coupling error from the estimated space and time 
errors, which are predefined and used efficiently in Algorithm~\ref{space_time_discrete_adaptive}. 
To this purpose,  at each iteration $k\geq 1$,
we define for all $1\leq n\leq N_{\textn{f}}$, 
the \textit{local spatial, temporal and iterative coupling estimators}
\bse\label{Disc_estimator_eq_sp_tm_cp}\begin{align}
  \label{Disc_estimator_eq_sp_P}
  \eta^{k,n}_{\textn{sp,P},K}
   & \eq \eta^{k,n}_{\textn{R,P},K}+||\vecw^{k,n}_{h}+\vK\nabla{\hat{p}^{k,n}}_{h}||_{\star,K},\\
    \label{Disc_estimator_eq_sp_U}
  \eta^{k,n}_{\textn{sp,U},K}
   & \eq \eta^{k,n}_{\textn{R,U},K}+\|\sigmabo^{k,n}_{h}-\sigmabo({\hat{p}^{k,n}}_{h},{\hat{\vecu}^{k,n}}_{h})\|_{K},\\
   \label{Disc_estimator_eq_tm_P}
   \eta^{k,n}_{\textn{tm,P},K}
   &\eq |||{\hat{p}^{k,n}}_{h}-{\hat{p}^{k,n-1}}_{h}|||_{K},\\
      \label{Disc_estimator_eq_tm_U}
   \eta^{k,n}_{\textn{tm,U},K}
   &\eq\|\vecsigma({\hat{p}}_{h}^{k,n},{\hat{\vecu}^{k,n}}_{h})-\sigmabo({\hat{p}}_{h}^{k,n-1},{\hat{\vecu}^{k,n-1}}_{h})\|_{K},\\
  \label{Disc_estimator_eq_it_P}
   \eta^{k,n}_{\textn{it,P},K}
  &\eq ||\vecw^{k,n}_{h}-\hat{\vecw}^{k,n}_{h}||_{\star,K},\\
  \label{Disc_estimator_eq_it_U}
   \eta^{k,n}_{\textn{it,U},K}
   &\eq \|\sigmabo^{k,n}_{h}-{\hat{\sigmabo}^{k,n}}_{h}\|_{K}.
   \end{align}\ese
   Therefrom, we  introduce  like  in~\eqref{globa_space_tim_estimators}, $\textn{a}=\textn{sp,\,tm,\,it}$,  
 \label{Disc_estimator_eq_SPLI}\begin{alignat}{3}
& \left(\eta^{k,n}_{\textn{a,J}}\right)^{2}\eq \int_{I^{n}_{\textn{f}}}\sum_{K\in\Tau_{h}}\left(\eta^{k,n}_{\textn{a,J},K}\right)^{2}\dt,\quad \textn{J}=\textn{P},\,\textn{U},
\end{alignat} 
and then introduce  
their global versions like in~\eqref{global_versions_estimators} by  
\begin{equation} \label{eq:eta:dist}\begin{split}
   \eta^{k}_{\textn{a,J}} \eq {} & \sqrt{\dfrac{L_{\textn{J}}}{2}}\Bigg(\Bigg\{\sum_{n=1}^{N_{\textn{f}}} \left(\eta^{k,n}_{\textn{a,J}}\right)^2\Bigg\}^{\frac 1 2} +\sqrt{2}\Bigg\{\sum_{n=1}^{N_{\textn{f}}}\tau^{n}_{\textn{f}}\sum_{l=1}^{n}\big(\eta^{k,l}_{\textn{a,J}}\big)^{2}\Bigg\}^{\frac 1 2}\\
  &\qquad+\sqrt{2}\Bigg\{\sum_{n=1}^{N_{\textn{f}}}\sum_{l=1}^{n}J_{nl}\Big(\sum_{q=1}^{l}\left(\eta^{k,q}_{\textn{a,J}}\right)^{2}\Big)\Bigg\}^{\frac 1 2}\Bigg)+\delta_{\textn{a}}\eta^{k}_{\textn{NC,J}} ,\quad \textn{J}=\textn{P},\,\textn{U},
\end{split}\end{equation}
where $\delta_{\textn{a}}=0$ for $\textn{a}=\textn{tm,\,it}$ and $\delta_{\textn{sp}}=1$. 
 \begin{rem}[Fixed-stress estimator]\label{rem:fs_estimator_nc}
 In the above estimators, we have to mention that, for conforming time discretization, the iterative coupling estimators $\eta^{k}_{\textn{it,J}}$, J=U, P, tends to zero when the fixed-stress algorithm converges. However, this is not true for non-conforming time discretization as in the multi-rate algorithm or the adaptive one. Precisely, at each iteration $k\geq 1$, the reconstructed flux and stress are satisfying respectively the local mass conservation~\eqref{localflowreconst} and \eqref{localstressreconst} which is not the case  for the approximate flux   from~\eqref{fstress_iterative_flow_multi-rate} and the approximate stress from~\eqref{fstress_iterative_mech_multi-rate} (step 2.(b)ii in Algorithm~\ref{multi-rate_discrete} and step 2.(b)iii in Algorithm~\ref{space_time_discrete_adaptive}) and this is even when the fixed-stress converges. In other words,
 the fixed-stress estimator $\eta^{k}_{\textn{it}}=\eta^{k}_{\textn{it,P}}+\eta^{k}_{\textn{it,U}}$  becomes a non-conformity in-time estimator when the fixed stress algorithm converges.
 \end{rem}
 
\begin{lem}[A posteriori error estimate distinguishing error components]\label{thm:error_distinguished}
Let the assumptions of Theorem~\ref{thm:main_theorem} be satisfied. Then there holds
\begin{equation}
  \label{eq:est_components}
  \|(p-\widetilde{p}^{k}_{\htau},\vecu-\widetilde{\vecu}^{k}_{\htau})\|_{\textn{en}} \leq 
  \sum_{\textn{J=P,U}}\{\underbrace{\eta^{k}_{\textn{sp,J}}+\eta^{k}_{\textn{tm,J}}}_{\eta^{k}_{\textn{disc,J}}}+\eta^{k}_{\textn{it,J}}\}.
\end{equation}
\end{lem}
\begin{proof} We  add and subtract the discrete fluxes in the flux estimator~\eqref{Disc_estimator_eq_P} then apply the triangle inequality yield, for all $K\in\Tau_ {h},\,\, t\in I_{\textn{f}}^{n}$, 
\begin{align}
  \label{Disc_estimator_eq_P_split1}  &\eta^{k,n}_{\textn{F,P},K}(t)\leq ||\hat{\vecw}^{k,n}_{h}+\vK\nabla{\hat{p}^{k,n}}_{h}||_{\star,K}+
  ||\vK(\nabla{\hat{p}}^{k}_{\htau}(t)-\nabla{\hat{p}^{k,n}}_{h})||_{\star,K}.
\end{align}
Another triangle inequality leads to
\begin{align} \label{Disc_estimator_eq_P_split2} 
&\eta^{k,n}_{\textn{F,P},K}(t)\leq 
 ||\vecw^{k,n}_{h}-\hat{\vecw}^{k,n}_{h}||_{\star,K}+||\vecw^{k,n}_{h}+\nabla{\hat{p}^{k,n}}_{h}||_{\star,K}+
  ||\vK(\nabla{\hat{p}}^{k}_{\htau}(t)-\nabla{\hat{p}^{k,n}}_{h})||_{\star,K}.
\end{align}
In the same way, we obtain for the stress  estimator~\eqref{Disc_estimator_eq_U}, for all $K\in\Tau_ {h},\,\, t\in I_{\textn{f}}^{n}$, 
\begin{align}
     \label{Disc_estimator_eq_U_split}
   &\eta^{k,n}_{\textn{F,U},K}(t)\leq\|\sigmabo^{k,n}_{h}-{\hat{\sigmabo}^{k,n}}_{h}\|_{K}+\|\sigmabo^{k,n}_{h}-\sigmabo({\hat{p}^{k,n}}_{h},{\hat{\vecu}^{k,n}}_{h})||_{K}+||\sigmabo({\hat{p}^{k}}_{\htau},{\hat{\vecu}^{k}}_{\htau})(t)-\sigmabo({\hat{p}^{k,n}}_{h},{\hat{\vecu}^{k,n}}_{h})||_{K}.
\end{align}
In these two inequalities, the first terms in the right-hand side form the  contributions of the pressure and displacement in the fixed-stress error, whereas the second  ones  contribute  to  the two components of the space  discretization error. The last  terms  can be  integrated in time, yielding to the two components of the time  error. What  remains is to   replace~\eqref{Disc_estimator_eq_P_split2} and \eqref{Disc_estimator_eq_U_split} 
in~\eqref{eq:est} with the  use  of~\eqref{Disc_estimator_eq_sp_tm_cp}--\eqref{eq:eta:dist}, where we used
the equality of norms in~\eqref{Disc_estimator_eq_P_split2}, i.e.,  $||\vK \nabla v||_{\star,K}=||\vK^{-\frac{1}{2}}(\vK \nabla v)||_{K}=|||v|||_{K}$,  leading to the estimate~\eqref{eq:est_components}.
\end{proof}

\section{Numerical results} 
\label{sec:NumericalResults} In this section we illustrate the efficiency of
our theoretical results on numerical experiments.
We have chosen two examples designed to show how the adaptive fixed-stress scheme 
 behaves vs the standard ones and this is done on  different physical and geometrical situations.
 
\subsection{Test problem 1: an academic example
with a manufactured solution}
We consider in the computational domain $\Omega=[0,1]^{2}$ and 
the final time $T=1$. The analytical solution of Biot's consolidation problem is supposed to be:
\begin{alignat*}{2}\label{first_exactsolution}
 &p(\vecx,t)\eq tx(1-x)y(1-y),\\
 &\vecw(\vecx,t)\eq-\vK\nabla p,\\
 &u_{1}(\vecx,t)=u_{2}(\vecx,t)\eq tx(1-x)y(1-y),\\
 &\sigmabo(\vecx,t)\eq\vectheta(\vecu)-\alpha p\II.
\end{alignat*}
The effective parameters are $\vK=\II$, and $\alpha=\mu=\lambda=1$. This analytical solution which has homogeneous initial and Dirichlet boundary values for $p$ and $\vecu$,  generates from~\eqref{Main_problem_model_mech}-~\eqref{Main_problem_model_cons} non-zeros  source terms $\vf(\vecx,t)$ 
and $g(\vecx,t)$.
\subsubsection{Stopping criteria balancing the error components}\label{subsection:stopcriter}
The aim here is to illustrate 
the performance of the adaptive  stopping criteria introduced in Section~\ref{sec:balanc_stopping}. {To this purpose, we consider a uniform space-time mesh  with  $h=1/16$, and  $\tau_{\textn{f}}^{n}=\tau_{\textn{m}}^{n}=(2h)^2$.} The tuning  parameter is chosen  $\beta=\dfrac{\alpha^2}{\delta(\frac{2\mu}{d}+\lambda)}$, with $\delta=2$. The choice of the parameter $\delta$  is   theoretically calculated in~\cite{BAUSE2017745,storvik2018optimization,Mikelic2014} and possibly should  lead to  the best performance of the   fixed-stress method in terms of number of iterations.  We first test the performance of the space--time fixed stress algorithm (Algo.~\ref{space_time_discrete}) equipped with an adaptive stopping criteria. Therein, we set $\gamma_{\textn{it}}=0.2$ and    compare the results with the standard approach
in which the fixed-stress algorithm is continued until the algebraic residual-based criteria~\eqref{class_stp_critera1} is satisfied for an  (arbitrary) threshold $\epsilon$.
 \begin{figure}[hbt]
    \centering
            \begin{subfigure}[b]{0.475\textwidth}
    \includegraphics[width=\textwidth]{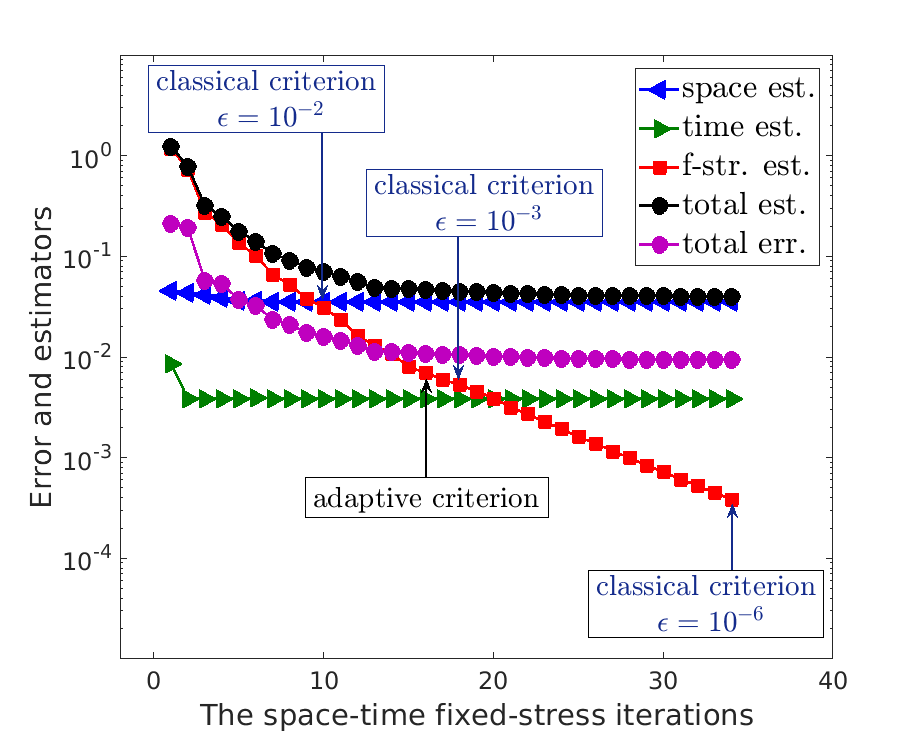}
    \caption{Error and estimators as a function of FS iterations}
        \label{fig:error_and_estimators}
    \end{subfigure}
    \begin{subfigure}[b]{0.475\textwidth}
    \includegraphics[width=\textwidth]{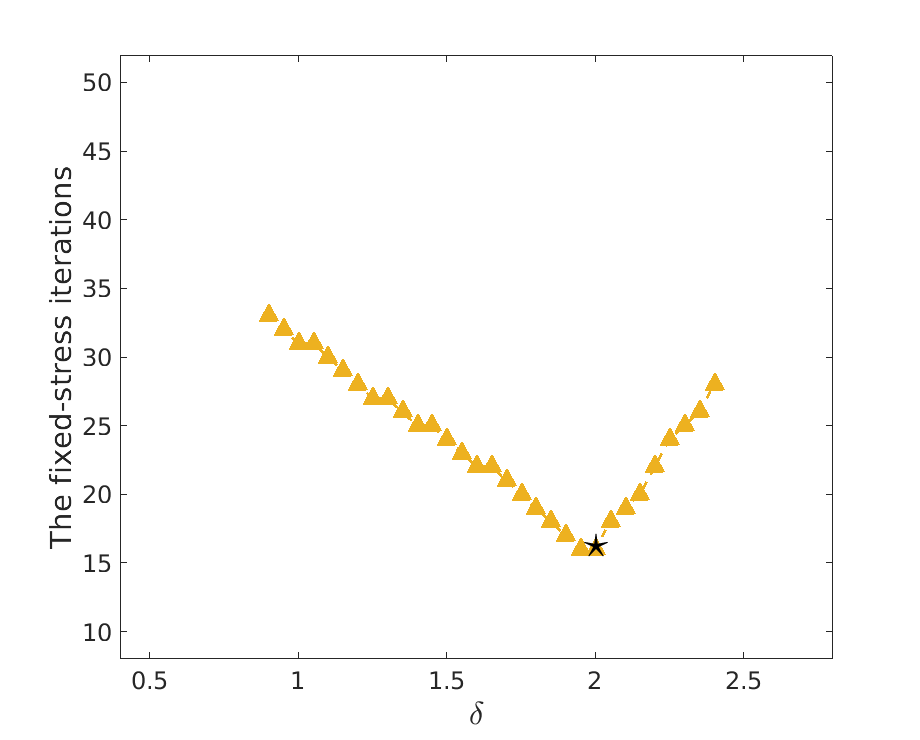}
    \caption{Number of FS iterations as a function of  $\delta$}
    \label{fig:robustness_vs_gamma}
    \end{subfigure}
    \caption{Number of fixed-stress iterations.}%
    \label{fig:stopping_criteria_space_time}
\end{figure}

Figure~\ref{fig:error_and_estimators}  displays the dependence of the total error and of the various estimators on the fixed-stress iterations, where   various stopping criteria are used.  We can see that the space (blue) and time (green) estimators remain constant during the computation in contrast to the fixed-stress estimator, which  gives a numerical indication that the error components are distinguished. For the fixed-stress estimator (red), we can see that  the adaptive stopping criteria~\eqref{Criteria_space_time} is satisfied after 16 iterations only, while the classical one~\eqref{class_stp_critera1} needs 34 iterations with $\epsilon=10^{-6}$, 18 iterations with $\epsilon=10^{-3}$, and only 10 iterations with $\epsilon=10^{-2}$. We can remark that the total error (magenta) and the total estimator (black) decrease rapidly for the first 12 fixed-stress iterations, after which  they decrease very slowly, as the influence of the fixed-stress iteration error
becomes negligible. This is exactly  the point where our adaptive fixed-stress method stops. This   results  in a significant saving of fixed-stress  iterations as well as excludes possible inaccurate solutions from the algorithm (like with $\epsilon=10^{-2})$.)  As an example, we make  a gain of $53\%$ of total fixed-stress iterations compared to the classical algorithm with $\epsilon=10^{-6}$.   
\begin{figure}[hbt]
    \centering
    \begin{subfigure}[b]{0.475\textwidth}
    \includegraphics[trim=0cm 2cm 0cm 5cm,clip=true,width=\textwidth]{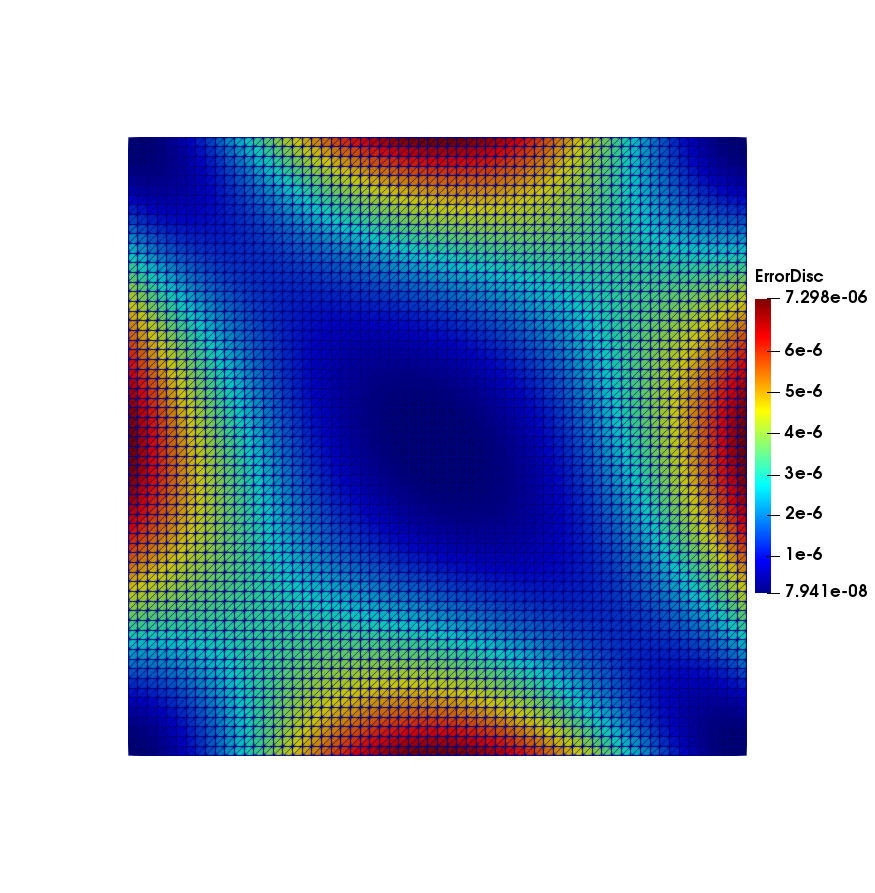}
    \caption{Discretization error}
        \label{fig:discretization_error}
    \end{subfigure}
    \begin{subfigure}[b]{0.475\textwidth}
    \includegraphics[trim=0cm 2cm 0cm 5cm,clip=true,width=\textwidth]{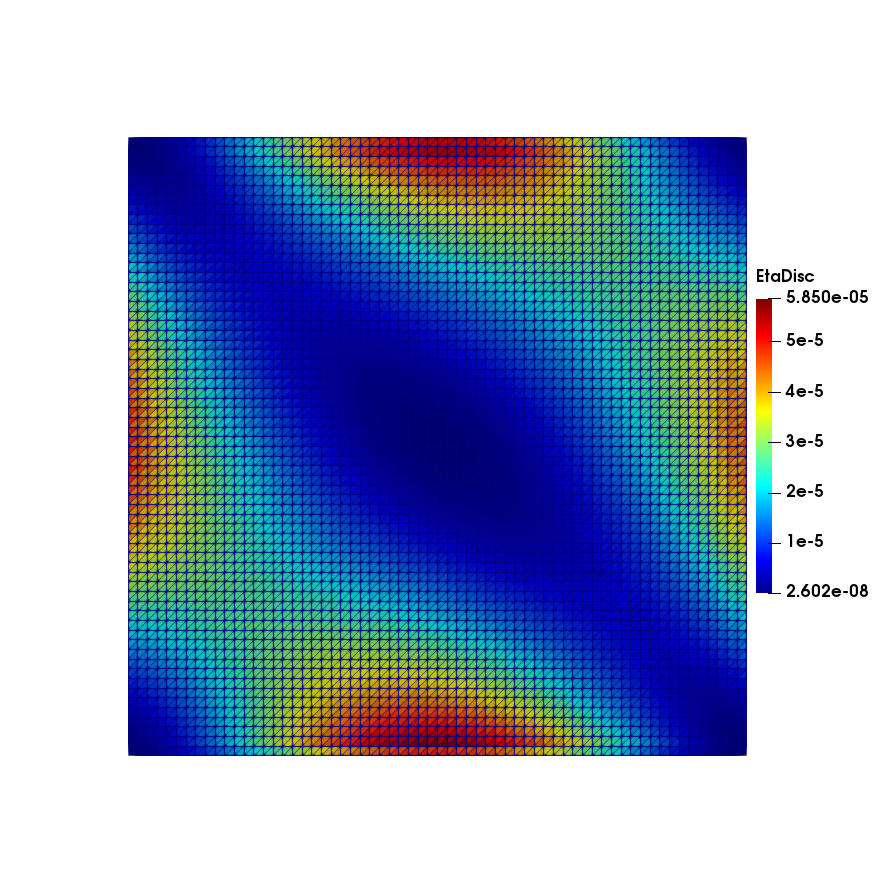}
        \caption{Discretization estimator}
      \label{fig:discretization_estimator}
    \end{subfigure}
    \begin{subfigure}[b]{0.475\textwidth}
    \includegraphics[trim=0cm 2cm 0cm 2cm,clip=true,width=\textwidth]{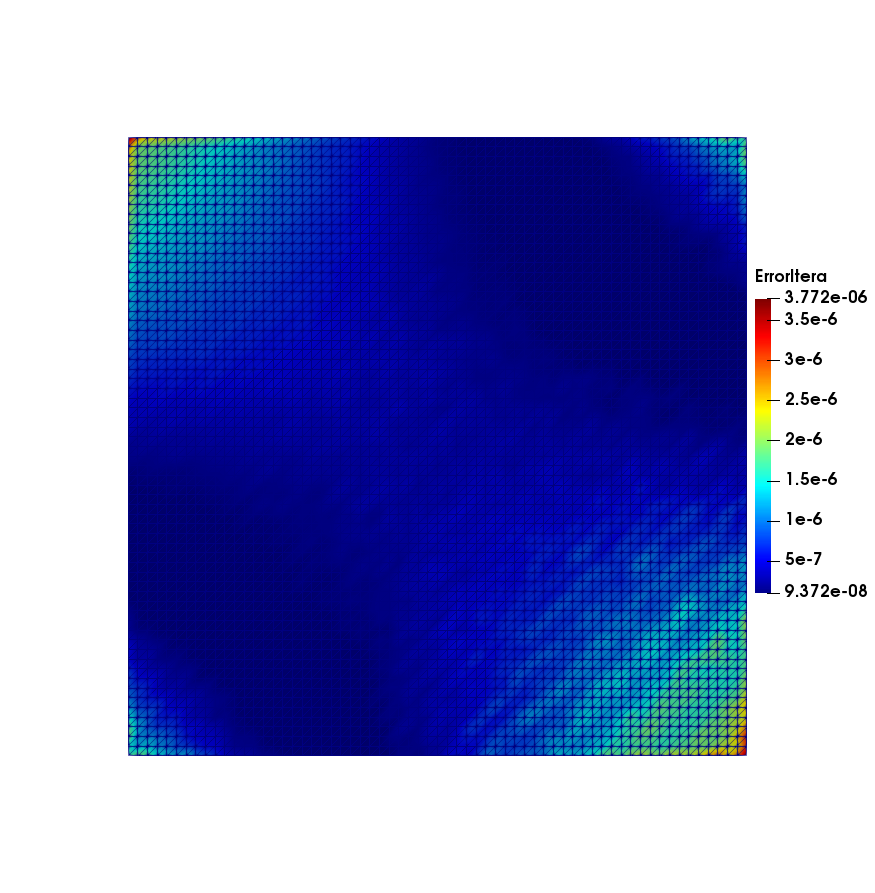}
    \caption{Fixed-stress error}
    \label{fig:fixedstress_error}
    \end{subfigure}
        \begin{subfigure}[b]{0.475\textwidth}
    \includegraphics[trim=0cm 2cm 0cm 2cm,clip=true,width=\textwidth]{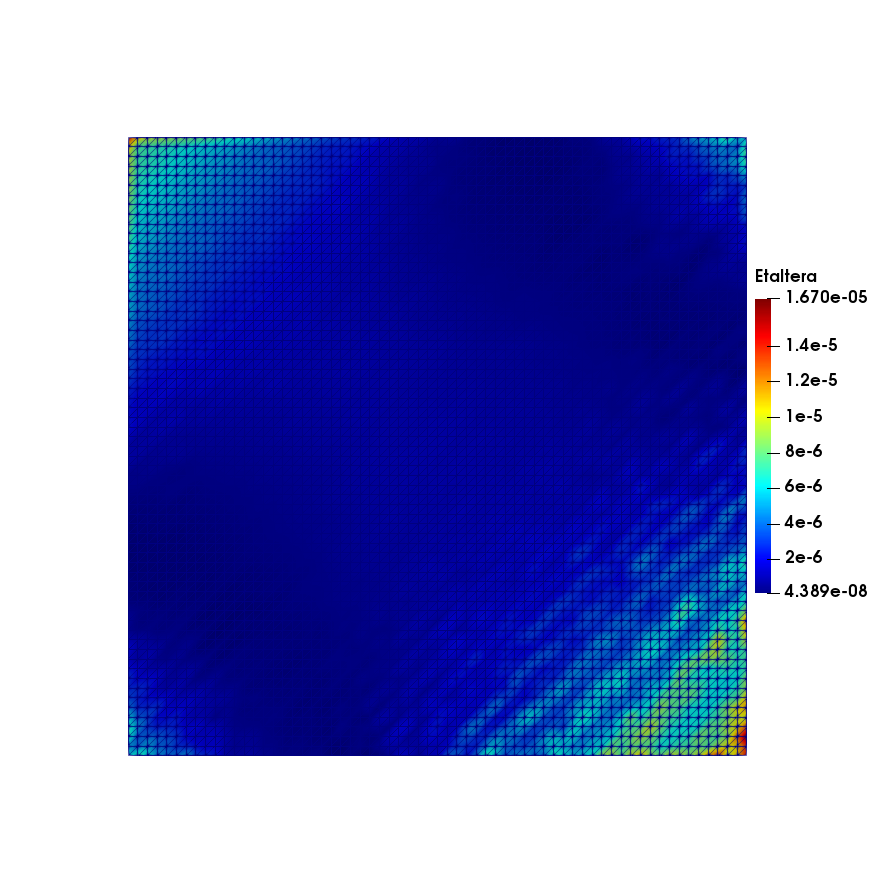}
    \caption{Fixed-stress estimator}
    \label{fig:fixedstress_estimator}
    \end{subfigure}
    \caption{Spatial distributions of the discretization and fixed-stress errors and of the corresponding estimates at $t=T$.}%
    \label{fig:results_case3}
\end{figure}

As known, the number of iterations to achieve convergence for the fixed-stress can differ considerably depending on the choice of the tuning parameter $\delta$. Thus, in Figure~\ref{fig:robustness_vs_gamma}, we plot the number of iterations required by the fixed-stress algorithm as function of the parameter $\delta$. Therein, we stop the  algorithm when the adaptive stopping criteria is satisfied. Clearly, the estimator behaves very similarly to what is usually observed for the fixed-stress error (see, e.g., \cite{storvik2018optimization}). Moreover, the theoretical parameter (marked by a star) coincides with   the numerically optimal value.  In Figure~\ref{fig:results_case3}, we display the  spatial distribution of the different error components (left) and of the corresponding estimators (right) at the final time $t=T$. Clearly, the distribution of the estimated errors reflects the exact ones. Also as expected, we  observe in~Figure~\ref{fig:fixedstress_estimator} that the estimated fixed-stress error is sufficiently small  to not contribute significantly in the overall error.


In Figure~\ref{fig:Effectivity index}, the effectivity index for the space--time fixed stress approach is presented. It is calculated by the ratio between the total estimator and the exact total  error at the iteration $k$ of the fixed-stress algorithm. The effectivity index oscillates during the first 10 iterations, then decreases to approximately 4.27 and then remains constant until the end of the computation. One of the reasons of why this factor  is far from 1, may be  that  the  negative norms involved in the  exact error $\|(p-\tilde{p}^{k}_{\htau},\vecu-\tilde{\vecu}^{k}_{\htau})\|_{\textn{en}}$ are not calculated. Another explanation is that  in practice, we use estimate~\eqref{eq:est_components} instead of~\eqref{eq:est} where the different error components are not yet separated.

 \begin{figure}[hbt]
    \centering
    \includegraphics[width=0.475\textwidth]{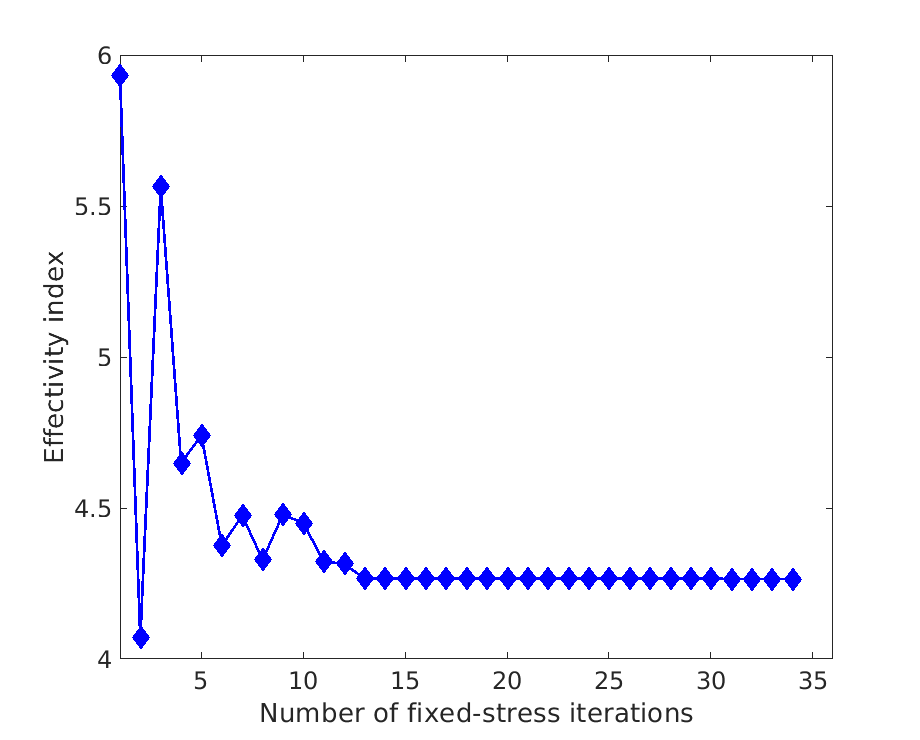}
    \caption{Effectivity index}
      \label{fig:Effectivity index}
\end{figure}

 In the second set of experiments, we study the performance of the adaptive stopping criteria on  the multi-rate fixed-stress algorithm (MFS). Here, we compare the results with the classical multi-rate algorithm in which the
 algebraic residual-based  criteria~\eqref{class_stp_critera2} is used with various threshold $\epsilon$.

\begin{figure}[hbt]
    \centering
    \begin{subfigure}[b]{0.475\textwidth}
    \includegraphics[width=\textwidth]{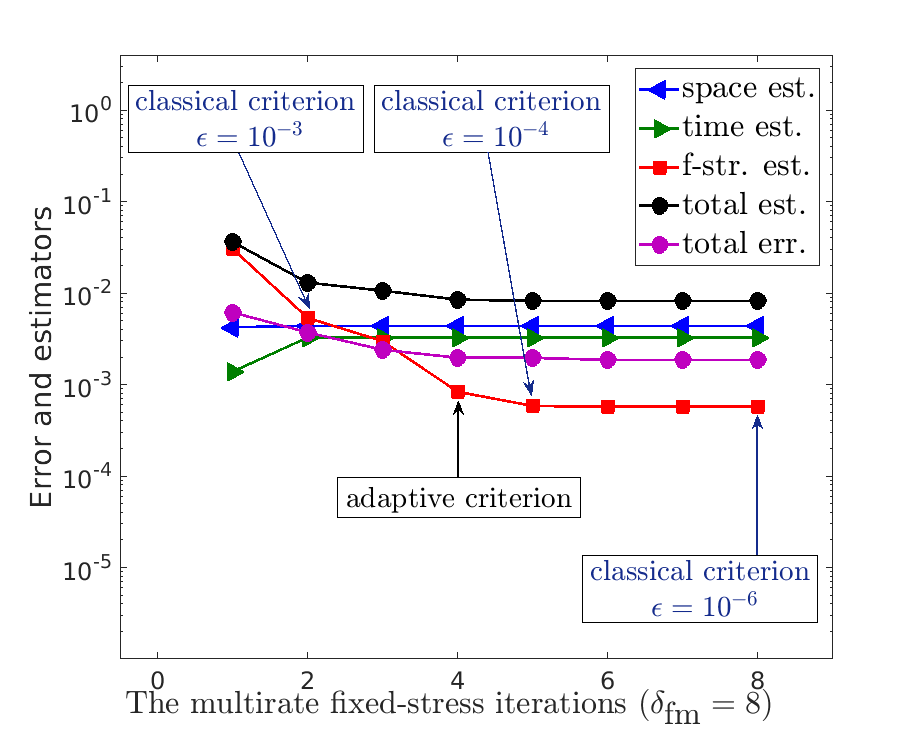}
       \caption{ Asynchronous time steps with  $\delta_{\textn{fm}}=8$}
           \label{fig:multi-rate_L_8}
    \end{subfigure}
        \begin{subfigure}[b]{0.475\textwidth}
    \includegraphics[width=\textwidth]{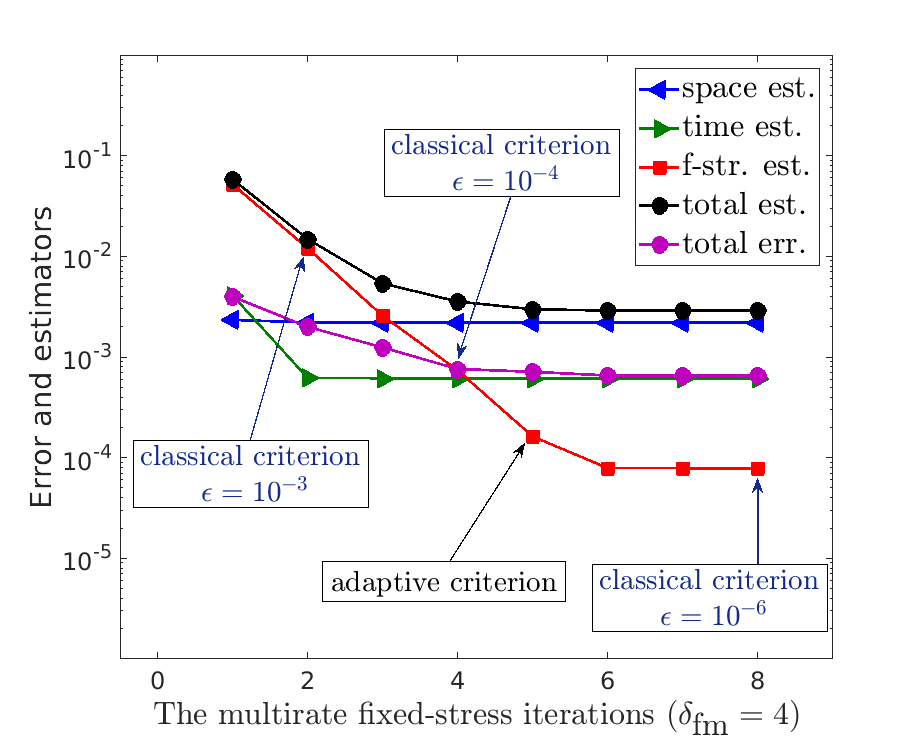}
      \caption{ Asynchronous time steps with $\delta_{\textn{fm}}=4$}
           \label{fig:multi-rate_L_4}
    \end{subfigure}
    \begin{subfigure}[b]{0.475\textwidth}
\includegraphics[width=\textwidth]{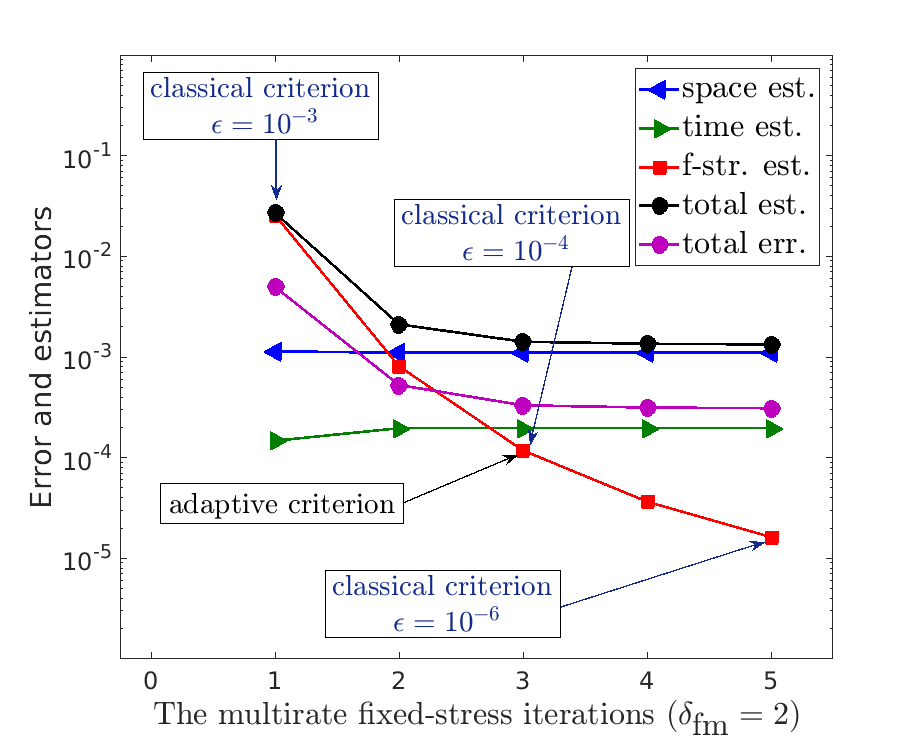}         
\caption{Asynchronous time steps with  $\delta_{\textn{fm}}=2$}
\label{fig:multi-rate_L_1}
\end{subfigure}
            \begin{subfigure}[b]{0.475\textwidth}
       \includegraphics[width=\textwidth]{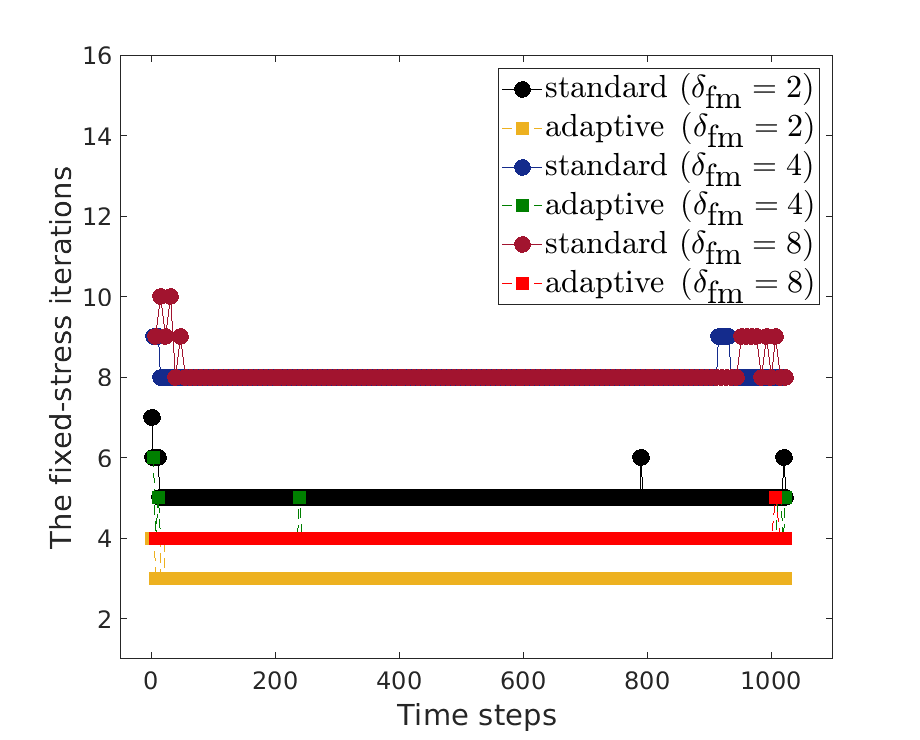}
        \caption{ Number of  iterations as a function of time.}
           \label{fig:multi-rate_cumulative}
    \end{subfigure}
    \caption{The standard and  adpative multi-rate fixed-stress  for various ratios
of discretization in time $\delta_{\textn{fm}}=8,\,4,\,2$.}%
    \label{fig:multi-rate}
\end{figure}
 
In Figures~\ref{fig:multi-rate_L_8}--\ref{fig:multi-rate_L_1}, we plot the evolution of the total error and the various estimators on the fixed-stress iterations for the final coarse mechanics step. There, we compare the adaptive to the standard multi-rate algorithm for various ratios of discretization in time,   $\delta_{\textn{fm}}=8,\,4,\,2$.  We  remark that (i) the discretization in space estimator (blue) as well as the discretization in time  estimator (green) are  approximately constant in each case  (ii)  the discretization in time estimator goes  up and approaching from   the discretization in space estimator when we increase the ratio   $\delta_{\textn{fm}}$. Note that the discretization in space estimator  is the same for the different ratios  if it is scaled with the time steps. These findings confirm numerically that we have practically distinguished the time discretization error from the spatial discretization error.  Concerning the fixed-stress estimator, we recall that in that case,  $\eta^{k,\ell}_{\textn{it}}$ mixes  fixed-stress  and nonconformity-in-time errors (see Remark~\ref{rem:fs_estimator_nc}). 
Thus,  we observe for the case $\delta_{\textn{fm}}=4$, and $\delta_{\textn{fm}}=8$ that, $\eta^{k,\ell}_{\textn{it}}$  dominates    the total error until  iteration 2 or 3, then becomes smaller than $\eta^{k,\ell}_{\textn{sp}}$ and $\eta^{k,\ell}_{\textn{tm}}$  until iteration 4 or 5, and therefrom  remains constant as the influence of the fixed-stress iteration error becomes negligible compared to the nonconformity-in-time error. For $ \delta_{\textn{fm}} = 2 $, the nonconformity error is small enough so as  not to contribute in $\eta^{k, \ell} _{\textn {it}} $ until convergence. For the cases, $\delta_{\textn{fm}}=4$ and $\delta_{\textn{fm}}=8$, we can see that  the classical multi-rate fixed-stress  equipped with \eqref{class_stp_critera2} as stopping criteria needs in the last coarse mechanics step   8 iterations to converge,   and   between 8 and 9 iterations for the  previous  ones. For $\delta_{\textn{fm}}=2$, the classical algorithm needs 5 iterations to converge, and between 5 and 7 for the previous ones.  For all the cases, the adaptive stopping criterion guarantees  that the fixed-stress algorithm converged to the correct solution
 compared to the classical criteria (see the results with  $\epsilon=10^{-3}$), while saving a substantial amount of computational effort (see the results with $\epsilon=10^{-6}$ and $10^{-4}$); see the  overall performance in the three cases depicted in Figure~\ref{fig:multi-rate_cumulative}  comparing the standard approaches with $\epsilon=10^{-6}$ to the adaptive ones.

\subsubsection{Adaptive time-stepping balancing the space and time errors}
In the second part of this test case, we verify   the impact of the  balancing criteria~\eqref{space_time_balance}-\eqref{flow__mech_time_balance} on the fixed-stress schemes. The balancing criteria~\eqref{space_time_balance} aims adapting the times steps for the flow and mechanics subsystems in such a way that   their spatial  and temporal error estimators \eqref{Disc_estimator_eq_SPLI} are equilibrated through the computation. This leads practically to  having $\eta^{n}_{\textn{sp}}\approx\eta^{n}_{\textn{tm}}$. That of the criteria~\eqref{flow__mech_time_balance} equilibrates the time errors from the mechanics and flow discretizations, i.e., $\eta^{n}_{\textn{tm,P}}\approx\eta^{n}_{\textn{tm,U}}$. Next, we will see that using the balancing~\eqref{space_time_balance} or~\eqref{flow__mech_time_balance} is   important  for the efficiency of the adaptive algorithm.

To this aim,  we compare on  three levels of uniform space-time mesh refinement, the standard space-time  (Algo.~\ref{space_time_discrete} ) and  the single- and multi-rate (Algo.~\ref{multi-rate_discrete}) with $\delta_{\textn{fm}}=8,\,4$  algorithms with the adaptive fixed-stress one~(Algo.~\ref{space_time_discrete_adaptive}). For the three refinement levels, we use the same  weights  $\gamma_{\textn{tm,J}}=0.8$ and $\Gamma_{\textn{tm,J}}=1.2$, $\textn{J=P,\,U}$. In Figure~\ref{fig:results_case4} (top), the ratio of the time discretization error over the space discretization error  from the flow (left) and from the mechanics (right)  as a function  of the total number of space--time unknowns is depicted for the aforementioned standard and adaptive algorithms. These results confirm numerically that we have  distinguished  the pressure and displacement errors as well as their time and space discretizations. Precisely, we can easily see that for the multi-rate schemes, the ratio $\frac{\eta_{\textn{tm,P}}}{\eta^{n}_{\textn{sp,P}}}$ (Figure~\ref{fig:results_case4} (top left)),  remains constant when changing the ratio $\delta_{\textn{fm}}$, in contrast to the ratio 
$\frac{\eta_{\textn{tm,U}}}{\eta^{n}_{\textn{sp,U}}}$ (Figure~\ref{fig:results_case4} (top right)), where the ratio increases with $\delta_{\textn{fm}}$. The effect of the resulting ratios on the overall estimate is shown in Figure~\ref{fig:results_case4} (bottom left). These results make it
evident  that the performance of the fixed-stress algorithms   is considerably improved  if they are equipped with the balancing criteria~\eqref{space_time_balance} (balanced1) or~\eqref{flow__mech_time_balance} (balanced2). Particularly,  the standard multi-rate algorithm reduces  the computational cost of the single-rate one, but  still much more expensive than the adaptive ones. In average, the adaptive one reduces the computational cost of the multi-rate one with $58\%$ while     the efficiency of the algorithm in terms of precision is much more preserved. In Figure~\ref{fig:results_case4} (bottom right), we chose the  third refinement level, then we plot the dependence of the total and  fixed-stress estimators as a function of  the fixed-stress iterations for the adaptive algorithms.  This result
confirms that  the algorithm is improved if it is equipped with the  balancing criteria \eqref{space_time_balance}  or~\eqref{flow__mech_time_balance}. Precisely, these balancing ensure that
the contribution of the fixed-stress estimator in the overall error becomes quickly  negligible (see Figure~\ref{fig:error_and_estimators} for the case without adaptivity), thus 
we can  stop the fixed-stress iterations by setting   $\gamma_{\textn{it}}=0.5$. Furthemore, with either of these balancing criteria we have   keeping a small  non-conformity in-time error which makes the application of adaptive stopping criteria more comfortable and guaranteed. Note that in the  standard  algorithms  (as shown in the results of subsection~\ref{subsection:stopcriter}), the time steps and the ratio $\delta_{\textn{fm}}$ are  mainly based on intuition and this may induce  an over-refinement in-time and may increase  the nonconformity-in-time errors, affecting considerably  the efficiency of the fixed-stress algorithm.  In Figure~\ref{fig:results_caseadapt}, we plot the pressure and displacement estimators as a function of the adaptive  time steps. Note that if the developed algorithm is equipped with asynchronous adaptivity  in space, we can significantly reduce the total computational cost, but also the total error, as  this later is dominated by the space error from the discretization of the flow subsystem. 


\begin{figure}[hbt]
    \centering
        \includegraphics[width=0.475\textwidth]{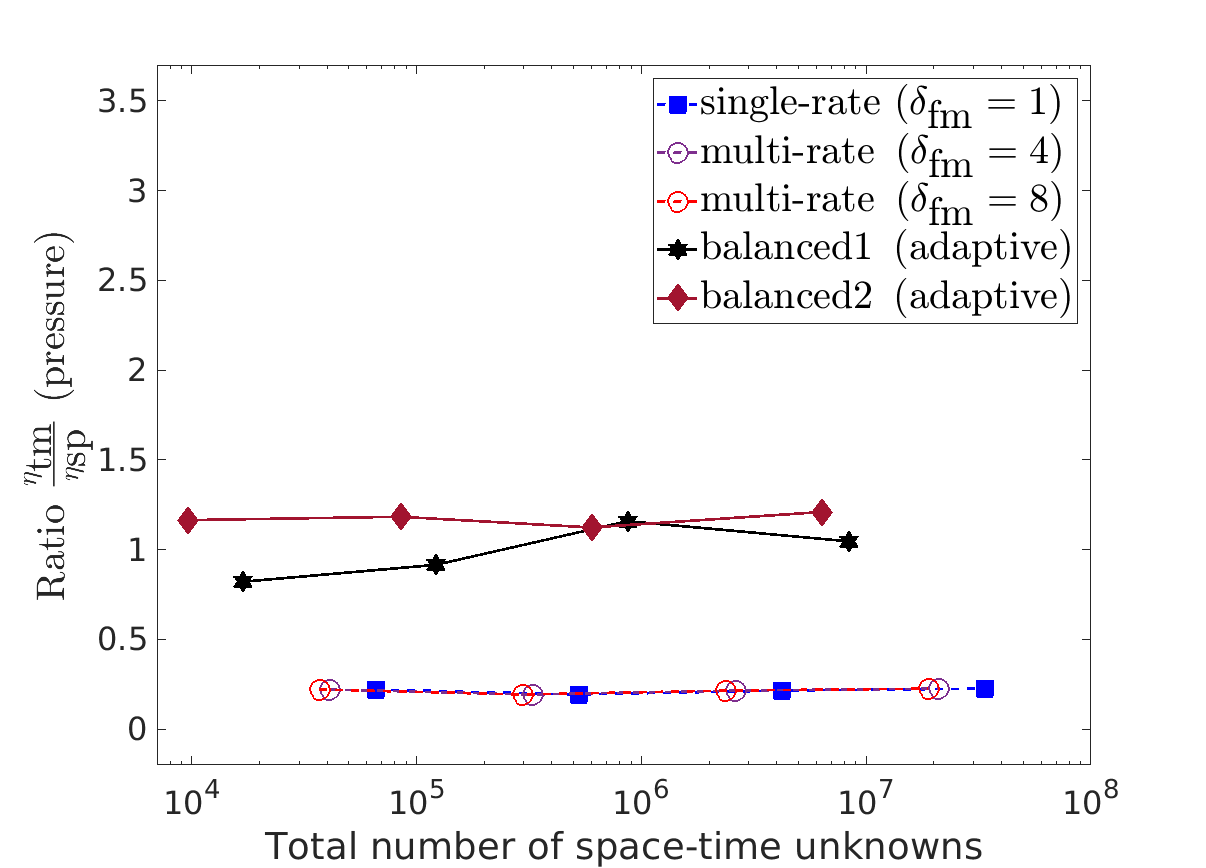}
        \includegraphics[width=0.475\textwidth]{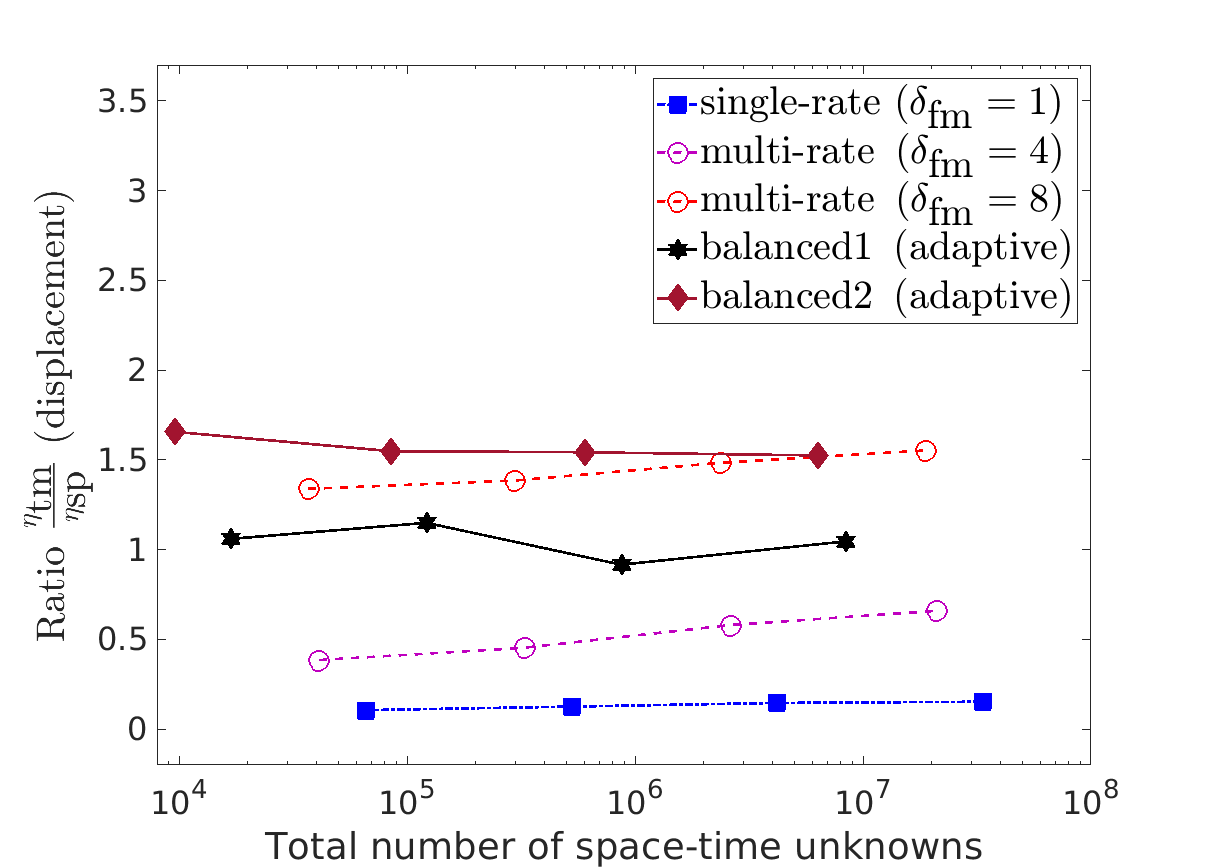}
    \includegraphics[width=0.475\textwidth]{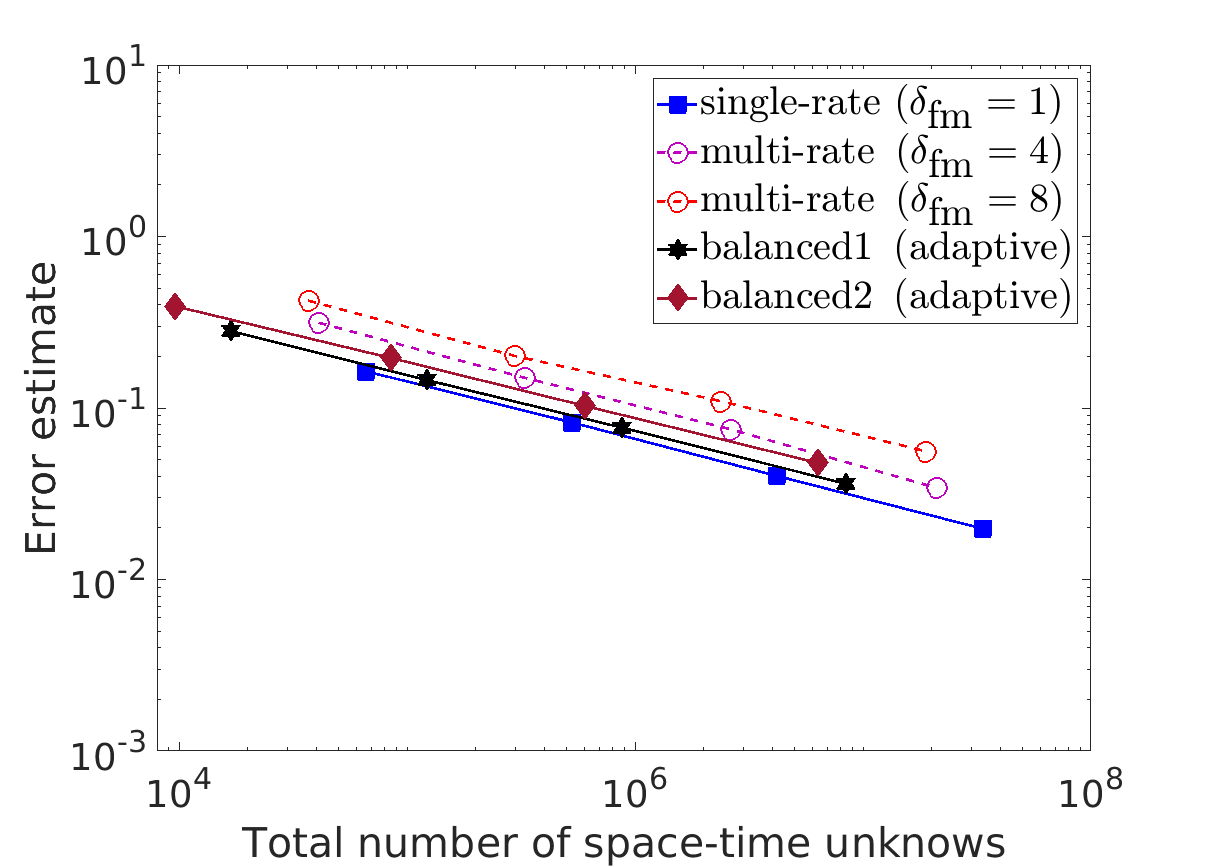}
    \includegraphics[width=0.475\textwidth]{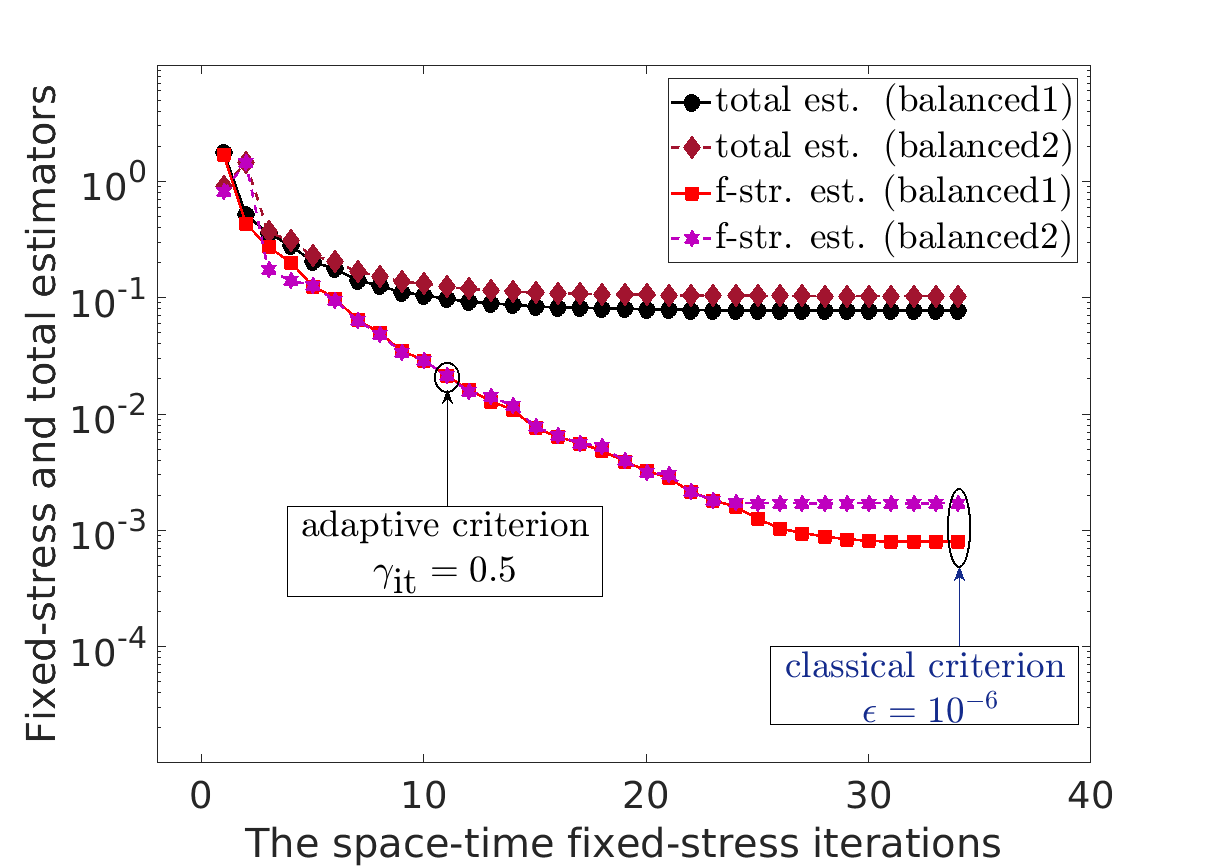}
    \caption{The  ratio  of  the  time  error  over  the  space  error  for the flow (top left) and mechanics (top right) problems as  a  function  of  the  total  number  of  space--time  unknowns. Comparison of the induced overall errors (bottom left). The total and fixed-stress estimators as a function of  fixed-stress iterations at the third refinement level (bottom right).
    }%
    \label{fig:results_case4}
\end{figure}

\begin{figure}[hbt]
    \centering
        \includegraphics[width=0.5\textwidth]{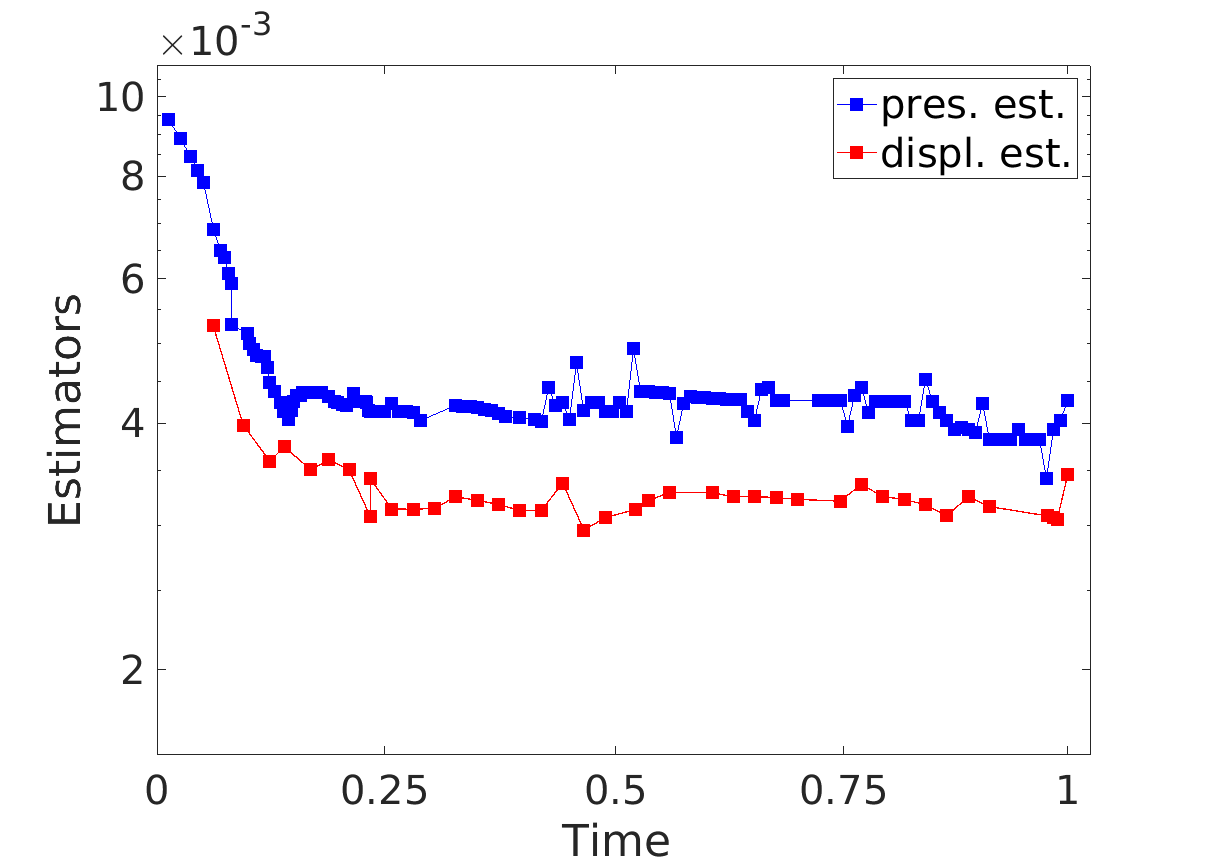}
    \caption{The  pressure and displacement estimators as a function of time.
    }%
    \label{fig:results_caseadapt}
\end{figure}
\subsection{Test problem 2: a poro-mechanical behavior of an osteonal tissue}
In this test case, the poroelastic model is carried out to
study the hydro-mechanical behavior of an idealized osteonal tissue.  This idealized structure is a group of osteons surrounded by their cement lines
and embedded in the interstitial bone matrix~\cite{NGUYEN2010384,Nguyen2011}.   The simplified  domain presents the parts of three different osteons connected by the  interstitial system (IS): a half osteon (O1) is located at the
bottom of the picture and two quarters of osteons (O2) and
(O3) are placed on the top-left and top-right corners, respectively (see Figure~\ref{fig:geo_and_mesh}). 
\begin{figure}[hbt]
    \centering
    \includegraphics[width=0.374\textwidth]{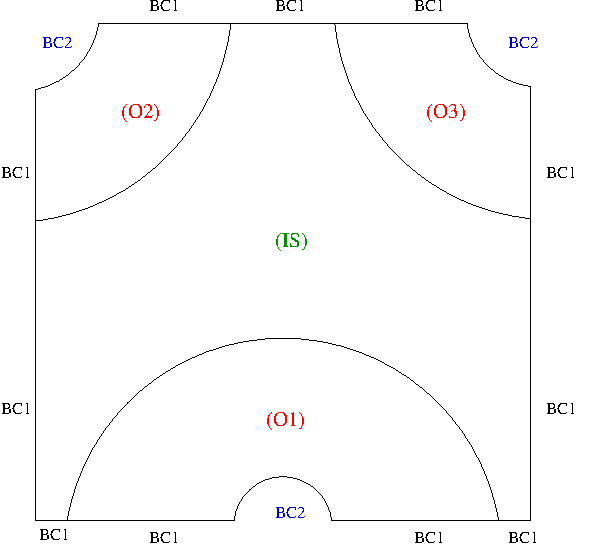}
        \includegraphics[width=0.485\textwidth]{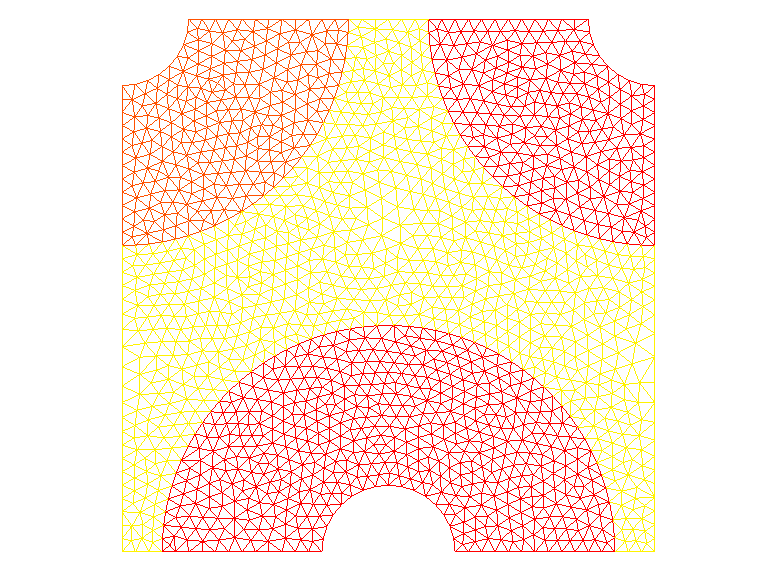}
    \caption{The computational domain (left) and associated mesh (right).
    }%
    \label{fig:geo_and_mesh}
\end{figure}
The used material properties  as generated in~\cite{nguyen2011influence} (see also \cite{doi:10.1080/10255840701479891}) are $\vK= 10^{-6}$ $(\textn{m}^{2})$  in the osteons and $\vK= 10^{-7}$ $(\textn{m}^{2})$ in the (IS)-domain.  The remaining parameters are $c_{0}=0.263$ $(\textn{GPa})$, $\alpha=0.132$ $(\textn{Kg.m}^{2})$,  $\mu=0.328$, and $\lambda=0.25$. The boundary conditions are $p=0$ and  $\vecsigma\vecn=0$ on the portion BC1 and 
$\vecu\cdot\vecn=0$ together with $(\vecsigma\vecn)\cdot\vectau=0$, and $\vecw\cdot\vecn=0$ on BC2. The final time is $T=15~(\mu \textn{s})$.

We use Algorithm~\ref{space_time_discrete_adaptive} equipped with~\eqref{space_time_balance}  where  we consider two computations that differ  by  the balancing parameters $\gamma_{\textn{tm,J}}$ and $\Gamma_{\textn{tm,J}}$.  We start with an initial time step $\tau_{\textn{f}}^{0}=2\cdot 10^{-3}~(\mu \textn{s})$, and $\tau_{\textn{m}}^{0}=4\tau_{\textn{f}}^{0}$. The  estimators are computed every 3 iterations to reduce the computational cost. Table~\ref{Table_1} compares 
the number  of space--time unknowns (number of asynchronous time steps, counting repetitions in the adaptive algorithm, fixed space unknowns) and performed  fixed-stress iterations, and the values of the error estimators  of the three  computations. We observe that the gain in the number of fixed-stress iterations as well as in the number of unknowns is significant. Indeed,  the two adaptive computations  need approximately 30 fixed-stress to converge while the standard fixed-stress algorithm needs  more than 132 iterations, thus, the  total computational cost  is     reduced of $88.6\%$ for  the first adaptive computation and of $82.5\%$ for  the second one. 
 
 To clarify this gain, we can observe in Figure~\ref{fig:fsvstotest} (left) that as soon as we perform $\approx30$ fixed-stress iterations, the fixed-stress estimator is sufficiently small to not contribute significantly on the overall error.  Also as expected, the adaptive stopping criterion stops the fixed-stress algorithm  when the solution is sufficiently accurate.  Figure~\ref{fig:fsvstotest} (left) confirms  also the role of the adaptivity in time, with which, the fixed-stress estimator becomes quickly smaller than the space and time discretization estimators, even with a large value of $\gamma_{\textn{it}}$, for example $\gamma_{\textn{it}}=0.5$.  We can also observe  that even with a small parameter $\gamma_{\textn{it}}=0.01$, the gain in the number of fixed-stress iterations is significant. In Figure~\ref{fig:fsvstotest} (right), we plot the pressure and displacement estimators as a function of time. Clearly, the displacement error dominates the pressure error along the simulation. In Figure~\ref{fig:results_case5}, we plot the approximate solution at the final time $t=T$.  Figure~\ref{fig:results_case6} compares the spatial  discretization errors for the pressure (top left) and displacement (top right), and the fixed-stress estimator (bottom), after using our adaptive stopping criteria at the final time $t=T$. Besides detecting the
dominating error at the circular boundary of the Osteons, we can  see that the total error is dominated by the mechanics discretization error, and that the fixed-stress estimator  is negligible.
 
  \begin{table}[h]
  \centering
\begin{tabular}{|c||c|c|c|c|c|}
\hline
Algorithm
    & \multicolumn{2}{c|}{adaptive} 
        & \multicolumn{2}{c|}{standard}\\ 
        \hline
 User-weights
    & $\gamma_{\textn{tm,J}}=0.9$, $\Gamma_{\textn{tm,J}}=1.1$    &   $\gamma_{\textn{tm,J}}=0.5$, $\Gamma_{\textn{tm,J}}=1.5$  &   \multicolumn{2}{c|}{ none }                 \\
    \hline
     Tolerance
    & $\gamma_{\textn{it}}=0.5$    &   $\gamma_{\textn{it}}=0.5$  &   \multicolumn{2}{c|}{ $\textnormal{err}_{\textn{FS}}^{k}\leq10^{-5}$ }                 \\
    \hline
     Nb. iterations
    &  $\approx 30 $   &  $\approx 30$   &   \multicolumn{2}{c|}{ $\approx 136$ }                 \\
    \hline
    Nb. unknowns
    &  571201    &  864341   &   \multicolumn{2}{c|}{ 1920375 }                 \\
    \hline
      Tot. estimate
    &  0.7943    &  0.723   &   \multicolumn{2}{c|}{ 0.638 }                 \\
    \hline
%
\end{tabular}
\caption{The  three computations in test problem~2.}
\label{Table_1}
\end{table}

\begin{figure}[hbt]
    \centering
        \includegraphics[width=0.475\textwidth]{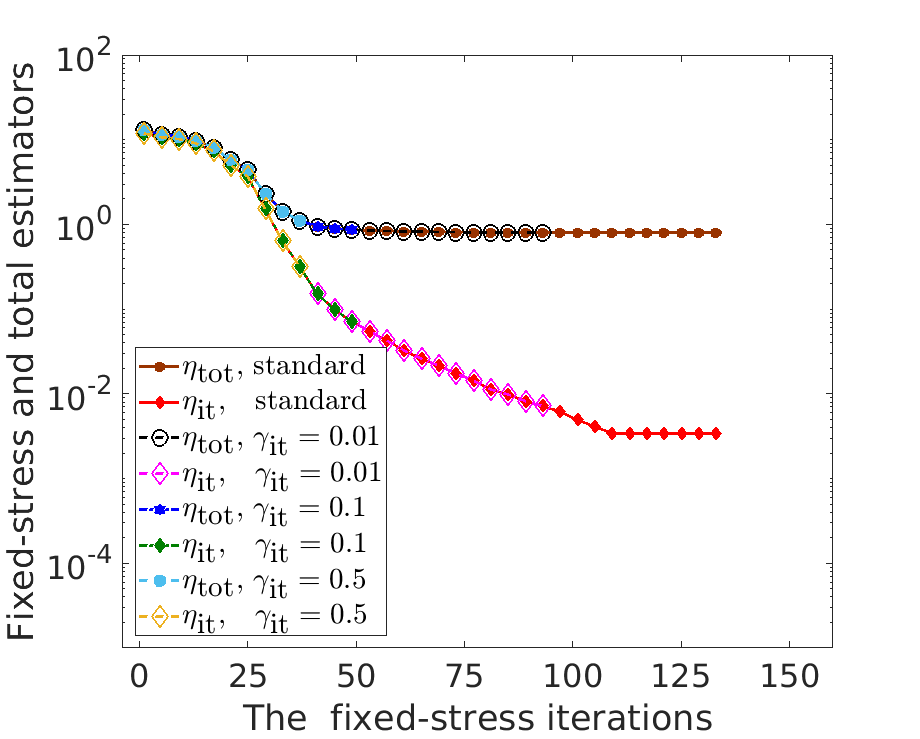}
            \includegraphics[width=0.475\textwidth]{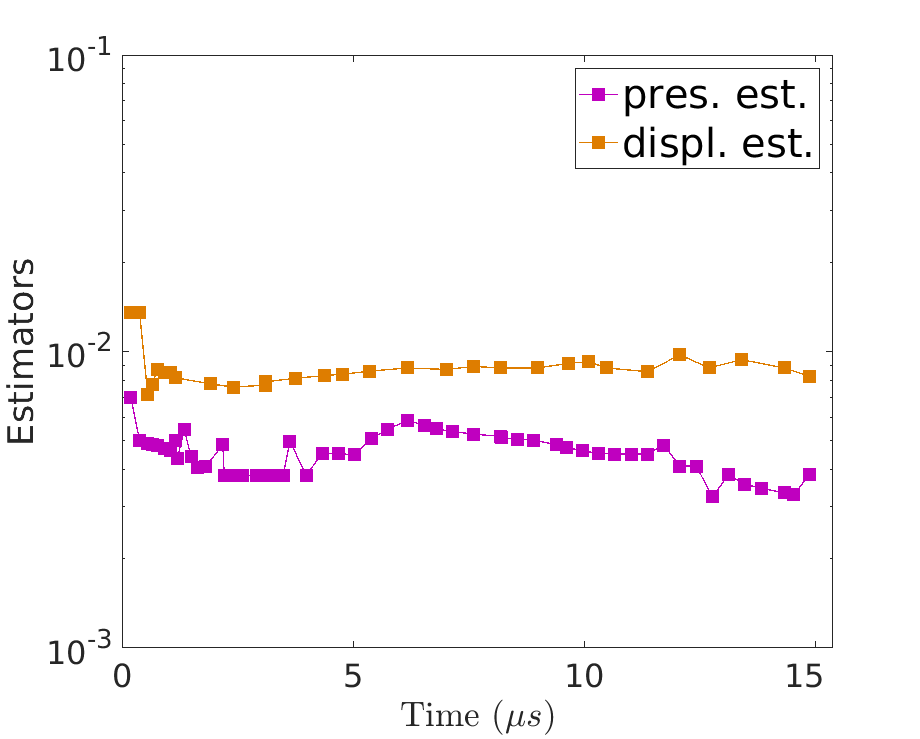}
    \caption{The fixed-stress and total estimators as a function of the fixed-stress iterations for various parameter $\gamma_{\textn{it}}$ (left). The pressure and displacement estimators as a function of time (right).}%
    \label{fig:fsvstotest}
\end{figure}

\begin{figure}[hbt]
    \centering
        \includegraphics[trim=0cm 5cm 0cm 5cm,clip=true,width=0.475\textwidth]{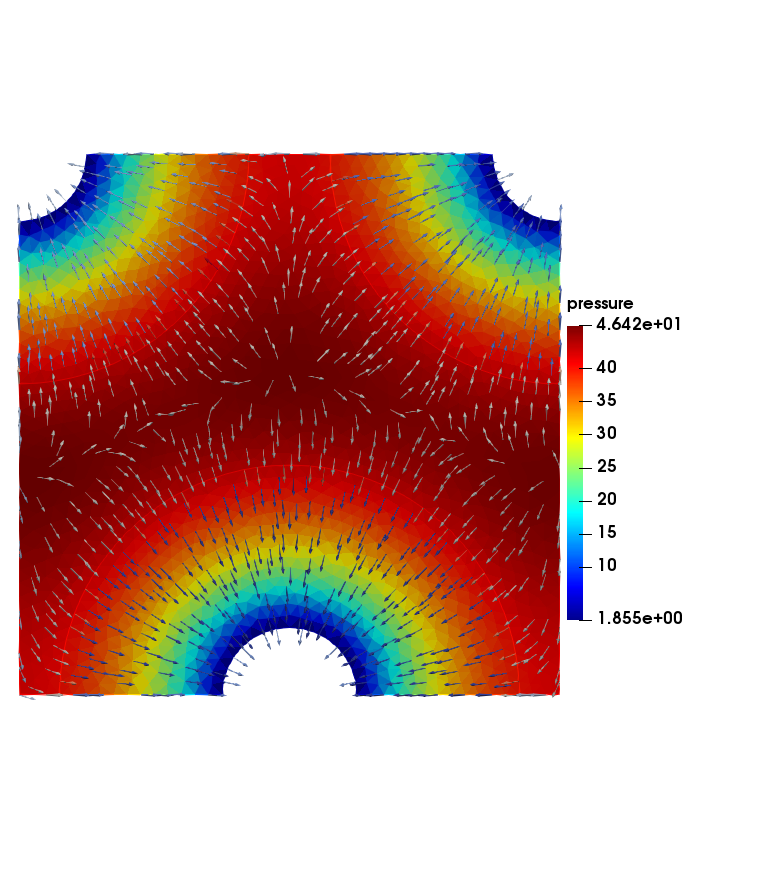}
        \includegraphics[trim=0cm 5cm 0cm 5cm,clip=true,width=0.475\textwidth]{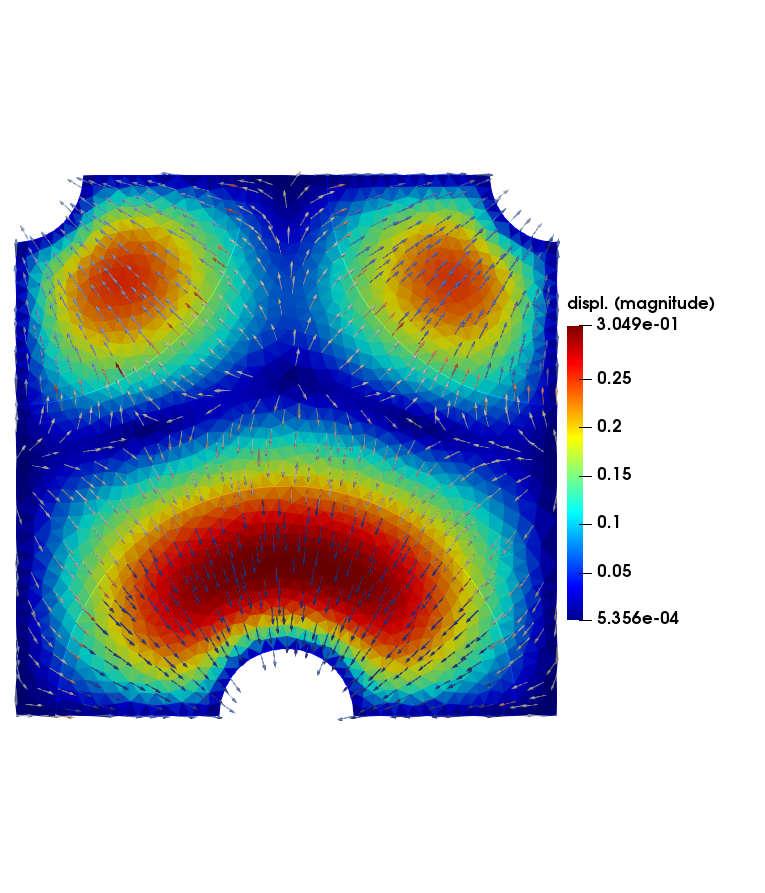}
    \caption{Approximate pressure and velocity (left) and displacement (right) at $t=T$.
    }%
    \label{fig:results_case5}
\end{figure}
\begin{figure}[hbt]
    \centering
      \begin{subfigure}[b]{0.475\textwidth}
    \includegraphics[trim=0cm 2cm 0cm 4cm,clip=true,width=\textwidth]{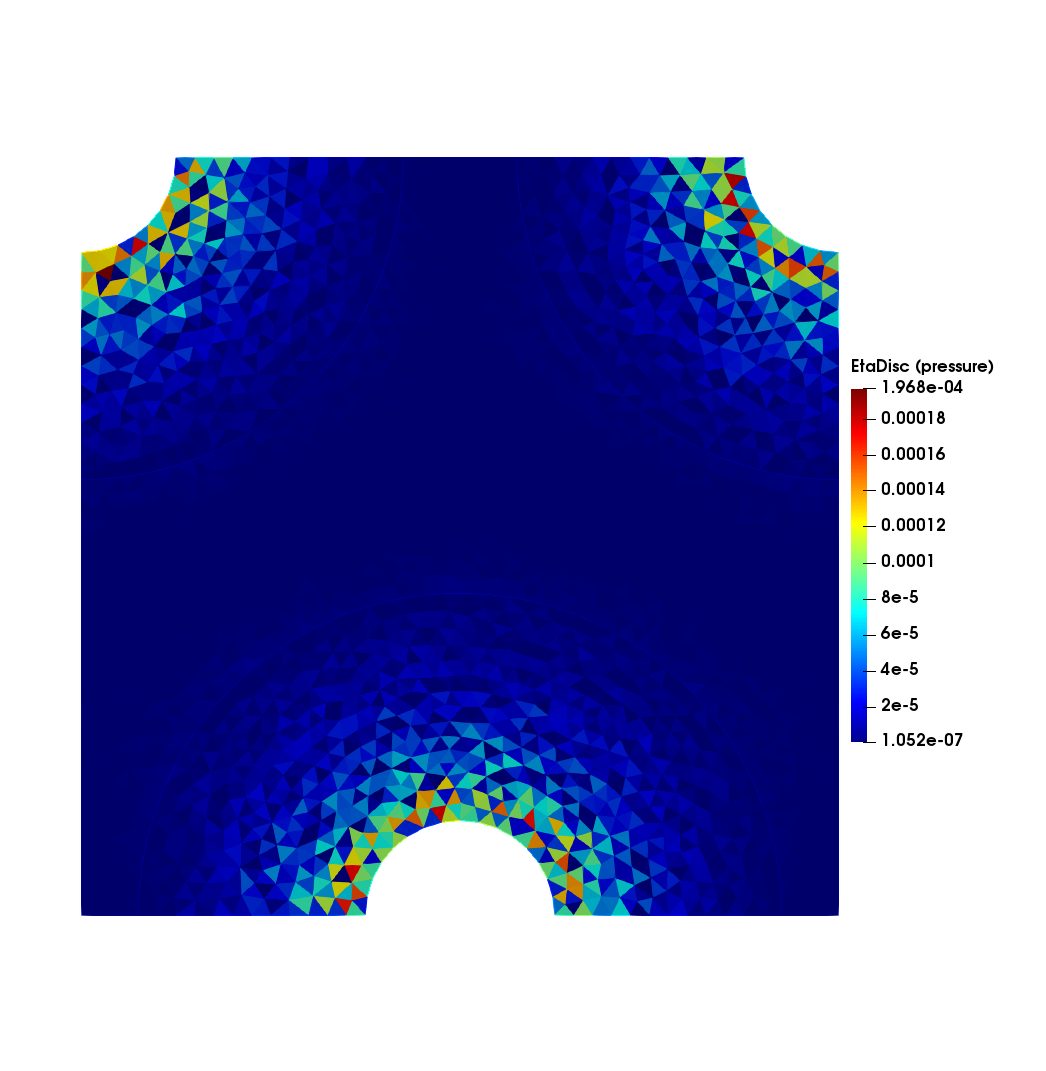}
    \caption{Pressure discretization estimator}
        \label{fig:discretization_error_pressure}
    \end{subfigure}
          \begin{subfigure}[b]{0.475\textwidth}
    \includegraphics[trim=0cm 2cm 0cm 4cm,clip=true,width=\textwidth]{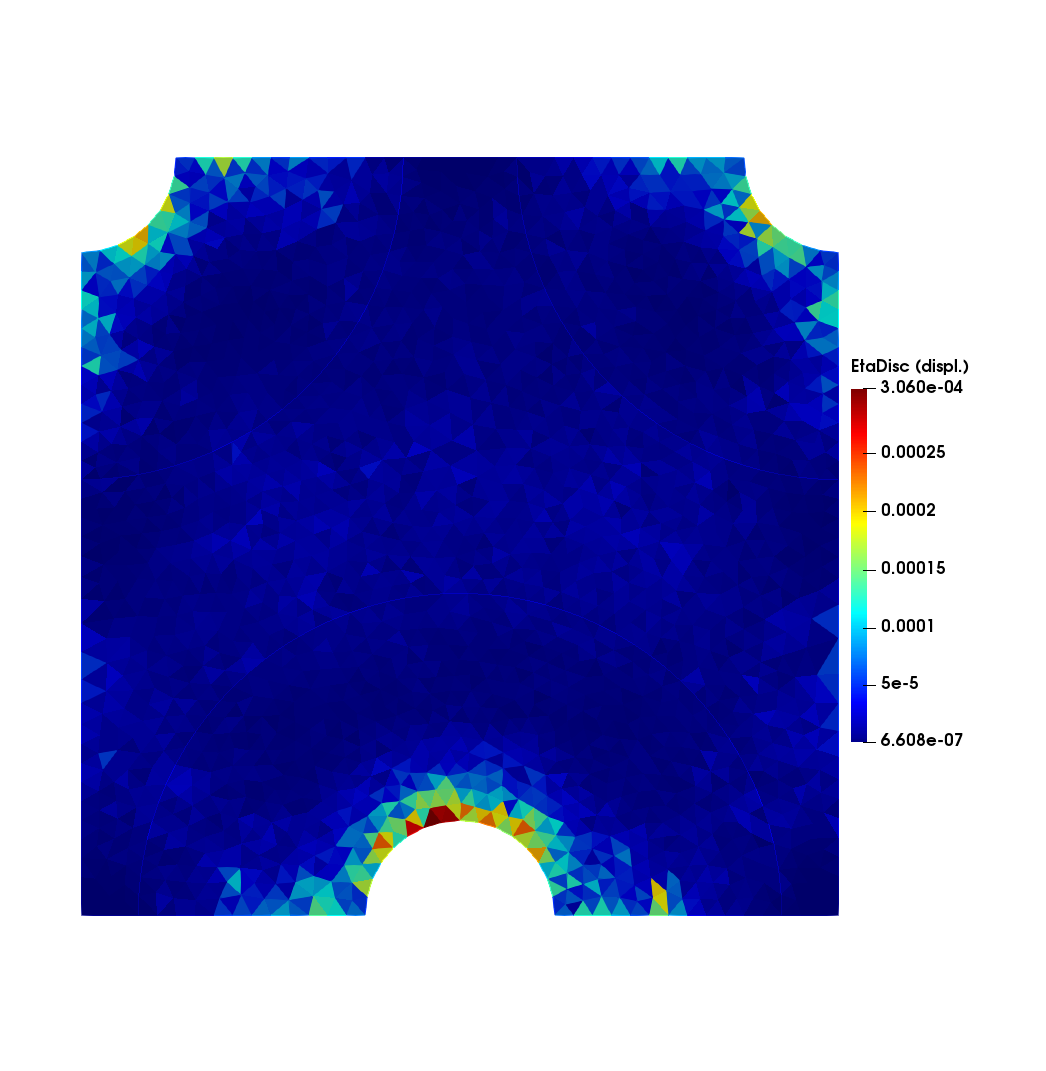}
    \caption{Displacement discretization estimator}
        \label{fig:discretization_error_displ}
    \end{subfigure}
     \begin{subfigure}[b]{0.485\textwidth}
       \includegraphics[trim=0cm 2cm 0cm 4cm,clip=true,width=\textwidth]{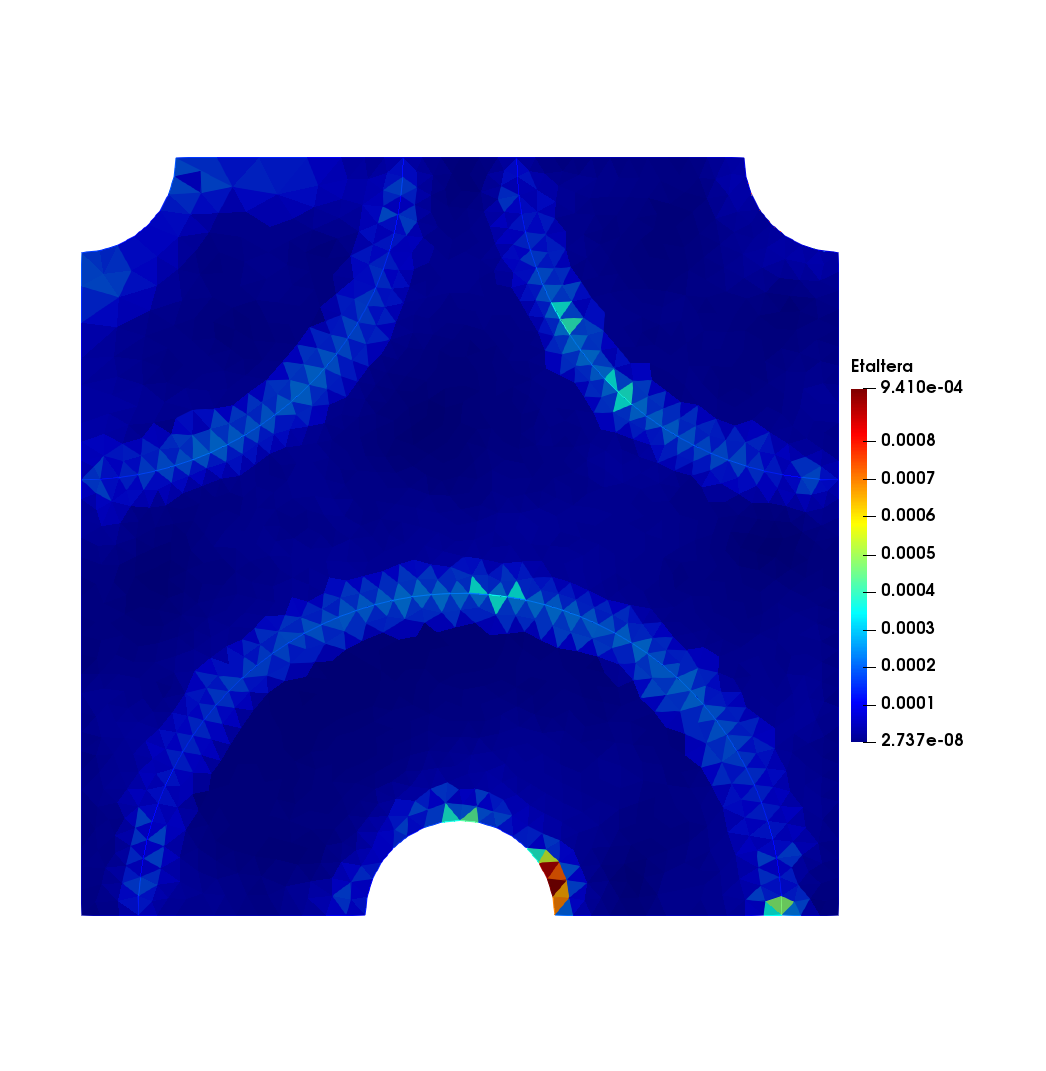}
    \caption{Fixed-stress estimator}
        \label{fig:discretization_error_fixedstress}
    \end{subfigure}
    \caption{Spatial distributions of the pressure and displacement discretization estimators and the  fixed-stress estimator at $t=T$.}%
    \label{fig:results_case6}
\end{figure}
\section{Conclusion}\label{sec:conclusion}
We proposed  in this paper adaptive fixed-stress iterative coupling schemes for the Biot system. Our adaptive algorithm can be used  either globally-in-time  or (partially)  via time windowing techniques,  and works as follows:
\begin{itemize}
 \item At the first iteration, both time step size of flow and mechanics
will be adapted in such a way that the space and time error contributions are equilibrated.
\item We then continue iterating, where several estimators (space, time and fixed-stress) are computed,  until 
the fixed-stress estimator becomes smaller (up to a user-chosen constant) than the other error components.
\end{itemize}
The numerical experiments demonstrated the accuracy of the estimated quantities while highlighting the
applicability of the presented adaptive algorithm. Particularly,  the algorithm  saves important number of iterations,  reduces significantly the total computational cost by adapting asynchronously the flow and mechanics time-steps and  avoiding over-in-time refinement  together with  maintaining a small  non-conformity in-time error. The algorithm may also help optimizing the tuning parameter. These benefits, together with the fact that we, a posteriori,  estimate the overall error that   is guaranteed and without unknown constant, leads to efficient and optimized adaptive fixed-stress coupling algorithm. Note that the present approach can be extended easily to other inexact coupling methods such as drained split, undrained split,  and fixed-strain split methods. Also, the present algorithm can be applied directly, without further developments to any flux- and stress-conforming discretizations of the flow and mechanics such that cell-centered finite volume or mimitic finite difference and can easily be extended to conforming methods.
\section*{Acknowledgment}
The research is supported  by the  Norwegian Research Council Toppforsk project  250223 (The TheMSES project: \href{project}{https://themses.w.uib.no}).
\vspace*{-0.2cm}

\enlargethispage{0.2cm}

\end{document}